\documentclass[11pt,a4paper]{scrartcl}
\usepackage[utf8]{inputenc}
\usepackage[english]{babel}
\usepackage{helvet, stmaryrd, ulem}
\usepackage{amsmath}
\usepackage{amsthm}
\usepackage{amsfonts}
\usepackage{amssymb}
\usepackage{mathtools}
\usepackage{mathrsfs}
\usepackage{dsfont}
\usepackage{tabularx}
\usepackage{xcolor}
\usepackage{bbm}
\usepackage[paper=a4paper,left=25mm,right=25mm,top=25mm,bottom=35mm]{geometry}

\usepackage{hyperref}
\usepackage[noabbrev,capitalize]{cleveref}

\author{Helena Kremp, Nicolas Perkowski}
\title{Kolmogorov equations with singular paracontrolled terminal conditions}

\newtheorem{theorem}{Theorem}[section]
\newtheorem{definition}[theorem]{Definition}     
\newtheorem{proposition}[theorem]{Proposition}
\newtheorem{lemma}[theorem]{Lemma}	

\newtheorem{corollary}[theorem]{Corollary}

\newtheorem{remark}[theorem]{Remark}
\newtheorem{assumption}[theorem]{Assumption}

\numberwithin{equation}{section}

\newcommand{\R}{\mathbb{R}}

\newcommand{\N}{\mathbb{N}}

\newcommand{\1}{\mathds{1}}
\newcommand{\F}{\mathcal{F}}
\newcommand{\La}{\mathfrak{L}^{\alpha}_{\nu}}
\let\oldmathcal\mathcal
\newcommand{\calC}{\mathscr{C}}
\newcommand{\CTcalC}{C_T\calC}
\newcommand{\calK}{\oldmathcal{K}}
\newcommand{\calX}{\mathscr{X}}
\newcommand{\calV}{\mathscr{V}}

\renewcommand{\mathcal}[1]{\mathscr{#1}}
\renewcommand{\epsilon}{\varepsilon}

\newcommand{\para}{\varolessthan}
\newcommand{\arap}{\varogreaterthan}
\newcommand{\reso}{\varodot}

\newcommand{\squeeze}[2][0]{\mbox{$\medmuskip=#1mu\displaystyle#2$}}

\DeclarePairedDelimiter{\abs}{\lvert}{\rvert}
\DeclarePairedDelimiter{\norm}{\lVert}{\rVert}

\DeclarePairedDelimiter{\paren}{(}{)}

\begin{document}
\begin{center}
\begin{huge}
Fractional Kolmogorov equations with singular paracontrolled terminal conditions\\
\end{huge}
\begin{Large}
\vspace{0.5cm}
Helena Kremp\footnote{Technische Universität Wien, helena.kremp@asc.tuwien.ac.at}, Nicolas Perkowski\footnote{Freie Universität Berlin, perkowski@math.fu-berlin.de}\\
\end{Large}
\vspace{1cm}
\hrule
\end{center}
We consider backward fractional Kolmogorov equations with singular Besov drift of low regularity and singular terminal conditions. To treat drifts beyond the socalled Young regime, we assume an enhancement assumption on the drift and consider paracontrolled terminal conditions.
Our work generalizes previous results on the equation from \cite{Cannizzaro2018,kp} to the case of \textit{singular paracontrolled terminal conditions} and simultaneously treats singular and non-singular data in one concise solution theory. We introduce a paracontrolled solution space, that implies parabolic time and space regularity on the solution without introducing the socalled "modified paraproduct" from \cite{Gubinelli2017KPZ}. The tools developed in this article apply for general linear PDEs that can be tackled with the paracontrolled ansatz.
\vspace{0,5cm}
\hrule
\vspace{1cm}

\begin{section}{Introduction}
\noindent Kolmogorov equations are second order parabolic differential equations. Their connection to stochastic processes was already investigated by Kolmogorov in the seminal work \cite{Kolmogorov1931}. There exist analytic and probabilistic methods to study Kolmogorov equations. We refer to the books \cite{Krylov1999, DaPrato2004, Krylov2008, bogachev2022fokker} for an overview on Kolmogorov equations in both finite and infinite dimensional spaces. In the finite dimensional setting, Kolmogorov equations with bounded and measurable coefficients and uniformly elliptic diffusion coefficients can be treated as a special case of the infinite dimensional Dirichlet form methods of \cite{Roeckner1995}, see also \cite[Section 2.4.1]{Krylov1999} and the connection to the martingale problem in \cite[Section 6.1.2]{Krylov1999}. We remain in the finite-dimensional setting, but consider distributional drifts in Besov spaces. Besov spaces play well with the  paracontrolled calculus that defines products of distributions, cf.~the Littlewood-Paley theory in \cite{Bahouri2011}. Previous articles that consider distributional drifts are \cite{Flandoli2003, Flandoli2017}, as well as \cite{Cannizzaro2018} in the setting of rougher distributional drifts. Heat kernel estimates for the solution to the Kolmogorov equation were established in \cite{Zhang2017, PvZ}.\\ In the article \cite{kp}, the Laplace operator is replaced by a generalized fractional Laplacian. We extend our previous results on the equation from \cite{kp} to allow for irregular terminal conditions. That is, we consider the fractional parabolic Kolmogorov backward equation
\begin{align*}
\paren[\big]{\partial_{t}-\La+V\cdot \nabla}u=f,\quad u(T,\cdot)=u^{T},
\end{align*} 
on $[0,T]\times\R^{d}$, where $\La$ generalizes the fractional Laplace operator $(-\Delta)^{\alpha/2}$ for $\alpha\in (1,2]$
and $V$ is a vector-valued Besov drift with negative regularity $\beta\in (\frac{2-2\alpha}{3},0)$, i.e.~$V\in C([0,T],(B^{\beta}_{\infty,\infty})^{d})=C_{T}\calC^{\beta}_{\R^{d}}$. Since $V$ is a distribution, we need to be careful with well-definedness of the product $V\cdot\nabla u$.
The regularity obtained from $-(-\Delta)^{\alpha/2}$ suggests that $u(t,\cdot) \in \mathcal C^{\alpha+\beta}$ if right-hand side $f$ and terminal condition $u^{T}$ are regular enough. Therefore we have $\nabla u(t,\cdot) \in \mathcal C^{\alpha+\beta-1}$. Since the product $V(t,\cdot) \cdot \nabla u(t,\cdot)$ is well-defined if and only if the sum of the regularities of the factors is strictly positive, we obtain the condition $\alpha+2\beta-1 >0$, equivalently $\beta > (1-\alpha)/2$. We call this the \textit{Young regime}, in analogy to the regularity requirements that are needed for the construction of the Young integral. However, we go beyond the Young regime, considering also the so-called \textit{rough regime} $\beta\in (\frac{2-2\alpha}{3},\frac{1-\alpha}{2}]$. In the rough case, we employ paracontrolled distributions (cf.~\cite{Gubinelli2015Paracontrolled}) to solve the equation. The idea is to gain some regularity by treating $u$ as a perturbation of the solution of the linearized equation with additive noise, $\partial_t w = \La w - V$. The techniques work as long as the nonlinearity $V \cdot \nabla u$ is of lower order than the linear operator $\La$, i.e. for $\alpha > 1$ or equivalently $(2-2\alpha)/3 < (1-\alpha)/2$. The price one has to pay to go beyond the Young regime is a stronger assumption on $V$. That is, we assume that certain resonant products involving $V$ are a priori given. Those play the role of the iterated integrals in rough paths theory (cf.~\cite{fh}). We then enhance $V$ by that resonant product component and call the enhancement $\calV$.\\
In \cite{Cannizzaro2018, kp}, only regular terminal conditions were considered, i.e.~$u^{T}\in\calC^{\alpha+\beta}$ in the Young regime and $u^{T}\in\calC^{2(\alpha+\beta)-1}$ in the rough regime. The right-hand side $f$ can either be an element of $C_{T}L^{\infty}$ or $f=V^{i}$ for $i=1,\dots,d$. There are techniques available to treat less regular terminal conditions, cf.~\cite[Section 6]{Gubinelli2017KPZ}. With the help of those techniques, one can allow for terminal conditions $u^{T}\in \calC^{(1-\gamma)\alpha+\beta}$ in the Young case and $u^{T}\in \calC^{(2-\gamma)\alpha+2\beta-1}$ in the rough case for $\gamma\in [0,1)$, obtaining a solution $u_{t}\in\calC^{\alpha+\beta}$ for $t<T$ and blow-up $\gamma$ for $t\to T$. In this work, consider moreover \textit{singular paracontrolled} right-hand sides $f$ as well as \textit{singular  paracontrolled terminal condition} $u^{T}$, which includes all cases mentioned above. Moreover, we can consider $f$ and $u^{T}$ more generally as elements of Besov spaces $\calC^{\theta}_{p}=B_{p,\infty}^{\theta}$ with integrability parameter $p\in[1,\infty]$. Examples for terminal conditions that we cover in the rough case include the Dirac measure, that is $u^{T}=\delta_{0}\in\calC^{0}_{1}$, and $u_{T}=V(T,\cdot)$. To be more precise, in the rough regime, we assume paracontrolled right-hand sides and terminal conditions, 
\begin{align*}
f=f^{\sharp} + f^{\prime}\para V,\quad u^{T}=u^{T,\sharp}+u^{T,\prime}\para V_{T},
\end{align*} with $u^{T,\prime}, f^{\prime}_{t}\in\calC^{\alpha+\beta-1}_{p}$ and remainders $f^{\sharp}_{t}\in\calC^{\alpha+2\beta-1}_{p}$, $u^{T,\sharp}\in\calC^{(2-\gamma)\alpha+2\beta-1}_{p}$ for $\gamma\in [0,1)$. For $f^{\prime}_t$ and $f^{\sharp}_t$  we also allow a blow-up $\gamma$ for $t\to T$. 
We prove existence and uniqueness of mild solutions of the Kolmogorov backward equation for singular paracontrolled data $(f,u^{T})$. The paracontrolled solution is an element of the solution space with blow-up $\gamma$ at terminal time $T$.
As a byproduct, we prove a new commutator estimate for the $(-\La)$-semigroup, cf.~\cref{lem:sharp}, that allows to gain not only space regularity, but also time regularity. Thanks to \cref{lem:sharp} there is no need for the so-called ``modified paraproduct" from \cite[Section 6.1]{Gubinelli2017KPZ}.  Moreover, we prove continuity of the Kolmogorov solution map
and a uniform bound for the solutions considered on subintervals of $ [0,T]$ for bounded sets of data.\\ 
It is important to mention that the techniques we develop in this article are not limited to that particular equation and can be used to treat other linear singular PDEs with the paracontrolled approach. In this sense we see the Kolmogorov PDE as a model example.\\
The work is structured as follows. In \cref{sec:prelim} we introduce the generalized fractional Laplacian $\La$ and its semigroup. We prove semigroup and commutator estimates. 
In \cref{section:singular-gen-eq} we introduce the solution spaces and prove generalized Schauder and commutator estimates thereon. 
Finally, we solve the Kolmogorov equation with singular paracontrolled data $(f,u^{T})$ in \cref{sec:singularKeq} and prove continuity of the solution map, as well as a uniform bound for the solutions on subintervals.\\[1em]

\end{section}

\begin{section}{Preliminaries}\label{sec:prelim}
Below we introduce some technical ingredients about Besov spaces and paraproducts, that we will need in the sequel. 
We study estimates for the generalized fractional Laplacian and its semigroup, as well as, commutator estimates involving the paraproducts and the fractional semigroup.\\

\noindent Let $(p_{j})_{j\geqslant -1}$ be a smooth dyadic partition of unity, i.e. a family of functions $p_{j}\in C^{\infty}_{c}(\R^{d})$ for $j\geqslant -1$, such that 
\begin{enumerate}
\item[1.)]$p_{-1}$ and $p_{0}$ are non-negative radial functions (they just depend on the absolute value of $x\in\R^{d}$), such that the support of $p_{-1}$ is contained in a ball and the support of $p_{0}$ is contained in an annulus;
\item[2.)]$p_{j}(x):=p_{0}(2^{-j}x)$, $x\in\R^{d}$, $j\geqslant 0$;
\item[3.)]$\sum_{j=-1}^{\infty}p_{j}(x)=1$ for every $x\in\R^{d}$; and
\item[4.)]$\operatorname{supp}(p_{i})\cap \operatorname{supp}(p_{j})=\emptyset$ for all $\abs{i-j}>1$.
\end{enumerate}
We then define the Besov spaces for $p,q\in [1,\infty]$,
\begin{align}\label{def:bs}
B^{\theta}_{p,q}:=\{u\in\mathcal{S}':\norm{u}_{B^{\theta}_{p,q}}=\norm[\big]{(2^{j\theta}\norm{\Delta_{j}u}_{L^{p}})_{j\geqslant -1}}_{\ell^{q}}<\infty\},
\end{align}
where $\Delta_{j}u=\mathcal{F}^{-1}(p_{j}\mathcal{F}u)$ are the Littlewood-Paley blocks, and the Fourier transform is defined with the normalization $\hat{\varphi}(y):=\F\varphi (y):=\int_{\R^{d}}\varphi(x)e^{-2\pi i\langle x,y\rangle}dx$ (and $\F^{-1}\varphi(x)=\hat{\varphi}(-x)$); moreover, $\mathcal{S}$ are the Schwartz functions and $\mathcal S'$ are the Schwartz distributions.\\
Let $C^{\infty}_{b}=C^{\infty}_{b}(\R^{d},\R)$ denote the space of bounded and smooth functions with bounded partial derivatives.
For $q=\infty$, the space $B^{\theta}_{p,\infty}$ has the unpleasant property that $C^{\infty}_b \subset B^\theta_{p,\infty}$ is not dense.
Therefore, we rather work with the following space:
\begin{align}\label{eq:sep}
B^{\theta}_{p,\infty}:=\{u\in\mathcal{S}'\mid\lim_{j\to\infty}2^{j\theta}\norm{\Delta_{j}u}_{L^{p}}=0\},
\end{align} 
for which $C^\infty_b$ is a dense subset (cf. \cite[Remark 2.75]{Bahouri2011}). We also use the notation $\calC^{\theta}_{\R^{d}}:=(\calC^{\theta})^{d}=\calC^{\theta}(\R^{d},\R^{d})$, $\calC^{\theta-}:=\bigcap_{\gamma<\theta}\calC^{\gamma}$ and $\calC^{\theta+}=\bigcup_{\gamma>\theta}\calC^{\gamma}$. Furthermore, we introduce the notation $\calC^{\theta}_{p}:=B^{\theta}_{p,\infty}$ for $\theta\in\R$ and $p\in [1,\infty]$, where $\calC^{\theta}:=\calC^{\theta}_{\infty}$ with norm denoted by $\norm{\cdot}_{\theta}:=\norm{\cdot}_{\calC^{\theta}}$.\\ 
For $1\leqslant p_{1}\leqslant p_{2}\leqslant\infty$, $1\leqslant q_{1}\leqslant q_{2}\leqslant\infty$ and $s\in\R$, the Besov space $B^{s}_{p_{1},q_{1}}$ is continuously embedded in $B_{p_{2},q_{2}}^{s-d(1/p_{1}-1/p_{2})}$ (cf. \cite[Proposition 2.71]{Bahouri2011}). Furthermore, we will use that for $u\in B^{s}_{p,q}$ and a multi-index $n\in\N^{d}$, $\norm{\partial^{n}u}_{B^{s-\abs{n}}_{p,q}}\lesssim \norm{u}_{B^{s}_{p,q}}$, which follows from the more general multiplier result from \cite[Proposition 2.78]{Bahouri2011}.\\
We recall from Bony's paraproduct theory (cf. \cite[Section 2]{Bahouri2011}) that in general for $u\in\calC^{\theta}$ and $v\in\calC^{\beta}$ with $\theta,\beta\in\R$, the product $u v:=u\para v+u\arap v +u \reso v$ , is  well defined in $\calC^{\min(\theta,\beta,\theta+\beta)}$ if and only if $\theta+\beta>0$. Denoting $S_{i}u=\sum_{j=-1}^{i-1}\Delta_{j}u$, the paraproducts are defined as follows
\begin{align*}
u\para v:=\sum_{i\geqslant -1} S_{i-1}u\Delta_{i}v,\quad u\arap v:=v\para u, \quad u\reso v:= \sum_{\abs{i-j}\leqslant 1}\Delta_{i}u\Delta_{j}v.
\end{align*} Here, we use the notation of \cite{Martin2017, Mourrat2017Dynamic} for the para- and resonant products $\para, \arap$ and  $\reso$.\\
In estimates we often use the notation $a\lesssim b$, which means, that there exists a constant $C>0$, such that $a\leqslant C b$. In the case that we want to stress the dependence of the constant $C(d)$ in the estimate on a parameter $d$, we write $a\lesssim_{d} b$.\\
The paraproducts satisfy the following estimates for $p,p_{1},p_{2}\in[1,\infty]$ with $\frac{1}{p}=\frac{1}{p_{1}}+\frac{1}{p_{2}}\leqslant 1$ and $\theta,\beta\in\R$ (cf. \cite[Theorem A.1]{PvZ} and \cite[Theorem 2.82, Theorem 2.85]{Bahouri2011})
\begin{equation}
\begin{aligned}\label{eq:paraproduct-estimates}
\norm{u\reso v}_{\calC^{\theta+\beta}_{p}} & \lesssim\norm{u}_{\calC^{\theta}_{p_{1}}}\norm{v}_{\calC^{\beta}_{p_{2}}}, \qquad \text{if }\theta +\beta > 0,\\
\norm{u\para v}_{\calC^{\beta}_{p}} \lesssim\norm{u}_{L^{p_{1}}}\norm{v}_{\calC^{\beta}_{p_{2}}}& \lesssim\norm{u}_{\calC^{\theta}_{p_{1}}}\norm{v}_{\calC^{\beta}_{p_{2}}}, \qquad \text{if } \theta > 0,\\
\norm{u\para v}_{\calC^{\beta+\theta}_{p}}& \lesssim\norm{u}_{\calC^{\theta}_{p_{1}}}\norm{v}_{\calC^{\beta}_{p_{2}}}, \qquad \text{if } \theta < 0.
\end{aligned}
\end{equation}
So if $\theta + \beta > 0$, we have $\norm{u v}_{\calC^{\gamma}_{p}}\lesssim\norm{u}_{\calC^{\theta}_{p_{1}}}\norm{v}_{\calC^{\beta}_{p_{2}}}$ for $\gamma:=\min(\theta,\beta,\theta+\beta)$.\\

\noindent We define the generalized fractional Laplacian $\La$ via Fourier analysis as follows.

\begin{definition}\label{def:fl}
Let $\alpha \in (0,2)$ and let $\nu$ be a symmetric (i.e. $\nu(A)=\nu(-A)$), finite and non-zero measure on the unit sphere $S\subset\R^{d}$. We define the operator $\La$ as
\begin{align}
\La\F^{-1}\varphi=\F^{-1}(\psi^{\alpha}_{\nu} \varphi)\qquad\text{for $\varphi\in C^\infty_b$,}
\end{align}  where
$\psi^{\alpha}_{\nu} (z):=\int_{S}\abs{\langle z,\xi\rangle}^{\alpha}\nu(d\xi).$
For $\alpha=2$, we set $\La:=-\frac{1}{2}\Delta$.
\end{definition}

\begin{remark}
If we take $\nu$ as a suitable multiple of the Lebesgue measure on the sphere, then $\psi^\alpha_\nu(z) = |2\pi z|^\alpha$ and thus $\La$ is the fractional Laplace operator $(-\Delta)^{\alpha/2}$. 
\end{remark}

\begin{assumption}\label{ass}
Throughout the paper, we assume that the measure $\nu$ from \cref{def:fl} has $d$-dimensional support, in the sense that the linear span of its support is $\R^d$. 
\end{assumption}
\noindent So far we defined $\La$ on $C^\infty_b$, so in particular on Schwartz functions. But the definition of $\La$ on Schwartz distributions by duality is problematic, because for $\alpha \in (0,2)$ the function $\psi^{\alpha}_{\nu}$ has a singularity in $0$. This motivates the next proposition.

\begin{proposition}[Continuity of the operator $\La$]\label{prop:contfl}
Let $\alpha\in (0,2]$. Then for $\beta\in\R$ and $u \in C^\infty_b$, $p\in[1,\infty]$, we have
\begin{align*}
\norm{\La u}_{\calC^{\beta-\alpha}_{p}}\lesssim\norm{u}_{\calC^{\beta}_{p}}.
\end{align*}
In particular, $\La$ can be uniquely extended to a continuous operator from $\calC^{\beta}_{p}$ to $\calC^{\beta-\alpha}_{p}$.
\end{proposition}

\begin{proof}
For $j \geqslant 0$ it follows from \cite[Lemma~2.2]{Bahouri2011} the estimate $\| \La \Delta_j u\|_{L^p} \lesssim 2^{-j(\beta-\alpha)} \|u \|_{\calC^{\beta}_{p}}$. This uses that $\psi^{\alpha}_{\nu}$ is infinitely continuously differentiable in $\R^{d}\setminus\{0\}$ with $\abs{\partial^{\mu}\psi^{\alpha}_{\nu}(z)}\lesssim\abs{z}^{\alpha-\abs{\mu}}$ for a multi-index $\mu\in\N_{0}^{d}$ with $\abs{\mu}:=\mu_{1}+\dots+\mu_{d}\leqslant\alpha$ and that $\Delta_{j}u$ has a Fourier transform, which is supported in $2^{j}\cal{A}$, where $\cal{A}$ is the annulus, where $p_{0}$ is supported. For $j=-1$ we use that $-\La \varphi = A\varphi$ for test functions $\varphi\in C^{\infty}_{b}$ and $A$ defined as
\begin{align}
A\varphi(x)=\int_{\R^{d}}\paren[\big]{\varphi(x+y)-\varphi(x)-\mathbf{1}_{\{\abs{y}\leqslant 1\}}(y) \nabla \varphi(x) \cdot y}\mu(dy)\qquad\text{for }\varphi\in C_{b}^{\infty}.
\end{align} and therefore
\begin{align}\label{bernstein-est}
	-\La \F^{-1}\tilde{p}_{-1}(x) & =  \int_{\R^d}\left(\F^{-1}\tilde{p}_{-1}(x+y) - \F^{-1}\tilde{p}_{-1} u(x) - \nabla \F^{-1}\tilde{p}_{-1} u(x)\cdot y \1_{\{|y| \leqslant 1\}} \right) \mu(dy)\nonumber \\
	& \lesssim \int_{B(0,1)} \|D^2\F^{-1}\tilde{p}_{-1}\|_{L^\infty}|y|^2 \mu(dy) + \|\F^{-1}\tilde{p}_{-1}\|_{L^\infty} \mu(B(0,1)^c) \nonumber\lesssim 1
\end{align}
where $B(0,1) = \{|y| \leqslant 1\}$ and $\tilde{p}_{-1}$ is smooth and compactly supported in a ball and such that $\tilde{p}_{-1}p_{-1}=p_{-1}$.
Then we obtain with $-\La\Delta_{-1}u=-\La\F^{-1}\tilde{p}_{-1}\ast\Delta_{-1}u$ and Young's convolution inequality,  
\begin{align*}
\norm{-\La\Delta_{-1}u}_{L^{p}}\leqslant \norm{-\La\F^{-1}\tilde{p}_{-1}}_{L^{1}}\norm{\Delta_{-1}u}_{L^{p}}\leqslant\norm{-\La\F^{-1}\tilde{p}_{-1}}_{L^{\infty}}\norm{\Delta_{-1}u}_{L^{p}}\lesssim \norm{u}_{\calC^{\beta}_{p}}.
\end{align*} 
\end{proof}

\noindent For $z \in \R^d \setminus\{0\}$, we also have
\[
	\psi^{\alpha}_{\nu}(z)= |z|^\alpha \int_{S} \Big|\Big\langle \frac{z}{|z|},\xi\Big\rangle\Big|^\alpha \nu(d\xi) \geqslant |z|^\alpha \min_{|y|=1} \int_{S} |\langle y,\xi\rangle|^\alpha \nu(d\xi),
\]
and by~\cref{ass} the minimum on the right hand side is strictly positive. Otherwise, there exists some $y_0\neq 0$ with $\int_S |\langle y_0,\xi\rangle|^\alpha \nu(d\xi) = 0$ and this would mean that the support of $\nu$ (and thus also its span) is contained in the orthogonal complement of $\operatorname{span}(y_0)$. Therefore, $e^{-\psi^\alpha_\nu}$ decays faster than any polynomial at infinity and outside of $0$ it even behaves like a Schwartz function.

\begin{lemma}[Semigroup estimates]\label{schauder}
Let $\nu$ be a finite, symmetric measure on the sphere $S\subset\R^{d}$ satisfying \cref{ass}. Let $P_{t}\varphi:=\mathcal{F}^{-1}(e^{-t\psi^{\alpha}_{\nu}}\hat{\varphi}) = \rho_t \ast \varphi$, where $t > 0$, $\rho_t = \mathcal F^{-1} e^{-t\psi^\alpha_\nu} \in L^1$, and $\varphi\in C^\infty_b$. Then we have for $\vartheta\geqslant 0$, $\beta\in\R$, $p\in[1,\infty]$
\begin{align}\label{eq:schauder1}
\norm{P_{t}\varphi}_{\calC^{\beta+\vartheta}_{p}}\lesssim (t^{-\vartheta/\alpha} \vee 1) \norm{\varphi}_{\calC^{\beta}_{p}},
\end{align} and for $\vartheta\in [0,\alpha]$
\begin{align}
\norm{(P_{t}-\operatorname{Id})\varphi}_{\calC^{\beta-\vartheta}_{p}}\lesssim t^{\vartheta/\alpha}\norm{\varphi}_{\calC^{\beta}_{p}}.\label{eq:schauder2}
\end{align}
Furthermore, for $\beta\in (0,1)$, $p=\infty$,
\begin{align}\label{eq:schauder-infty}
\norm{(P_{t}-\operatorname{Id})\varphi}_{L^{\infty}}\lesssim t^{\beta/\alpha}\norm{\varphi}_{\calC^{\beta}}.
\end{align}
Therefore, if $\vartheta\geqslant 0$, then $P_{t}$ has a unique extension to a bounded linear operator in $L(\mathcal{C}^{\beta},\mathcal{C}^{\beta+\vartheta})$ and this extension satisfies the same bounds. 
\end{lemma}

\begin{proof}
In the case $\theta\in [0,\alpha)$, this follows from~\cite[Lemma~A.5]{Gubinelli2015Paracontrolled}, see also \cite[Lemma~A.7]{Gubinelli2015Paracontrolled}, whose generalization to integrability $p\in[1,\infty]$ is immediate. For the case $\vartheta=\alpha$ in \eqref{eq:schauder2}, we estimate
\begin{align*}
\norm{(P_{t}-\operatorname{Id})\varphi}_{\calC^{\beta-\alpha}_{p}}&=\norm[\bigg]{\int_{0}^{t}(-\La)P_{r}\varphi dr}_{\calC^{\beta-\alpha}_{p}}\\&\leqslant\int_{0}^{t}\norm{(-\La)P_{r}\varphi}_{\calC^{\beta-\alpha}_{p}}dr\\&\lesssim\int_{0}^{t}\norm{P_{r}\varphi}_{\calC^{\beta}_{p}}dr\lesssim t\norm{\varphi}_{\calC^{\beta}_{p}}
\end{align*} using \cref{prop:contfl} and \eqref{eq:schauder1} for $\vartheta=0$. \eqref{eq:schauder-infty} follows from \cite[Lemma~A.8]{Gubinelli2015Paracontrolled}.
\end{proof}
\noindent The next three lemmas deal with commutators between the $(-\La)$ operator, its semigroup and the paraproduct. The proofs can be found in \cref{Appendix A}. 
\begin{lemma}\label{lem:La-comm}
Let $\alpha\in (1,2]$, $f\in \calC^{\sigma}_{p}$ and $g\in\calC^{\varsigma}$ with $\sigma\in (0,1)$ and $\varsigma\in\R$, $p\in[1,\infty]$. Then the commutator for $(-\La)$ follows:
\begin{align*}
\norm{(-\La)(f\para g)-f\para(-\La)g}_{\calC^{\sigma+\varsigma-\alpha}_{p}}\lesssim\norm{f}_{\calC^{\sigma}_{p}}\norm{g}_{\calC^{\varsigma}}.
\end{align*}
\end{lemma}

\begin{lemma}\label{schaudercom}
Let $(P_{t})$ be as in \cref{schauder}.
Then, for $\sigma\in (0,1)$, $\varsigma\in\R$, $p\in[1,\infty]$ and $\vartheta \geqslant -\alpha$ the following commutator estimate holds true:
\begin{align}
\norm{P_{t}(u\para v)-u\para P_{t}v}_{\calC^{\sigma+\varsigma+\vartheta}_{p}}\lesssim (t^{-\vartheta/\alpha}\vee 1)\norm{u}_{\calC^{\sigma}_{p}}\norm{v}_{\calC^{\varsigma}}.\label{eq:scom}
\end{align}
\end{lemma}

\begin{lemma}\label{lem:a1}
Let $\La$ and $(P_{t})_{t\geqslant 0}$ be defined as in \cref{def:fl} and \cref{schauder} and let $\alpha\in (1,2]$. Let $T>0$, $\sigma\in (0,1)$, $\varsigma\in\R$, $p\in[1,\infty]$ and $\theta\geqslant 0$. Then the commutator on the operator $(-\La)P_{t}$ follows:
\begin{align*}
\norm{(-\La)P_{t}(u\para v)-u\para (-\La)P_{t}v}_{\calC^{\sigma+\varsigma-\alpha+\theta}_{p}}\lesssim (t^{-\theta/\alpha}\vee 1)\norm{u}_{\calC^{\sigma}_{p}}\norm{v}_{\calC^{\varsigma}}.
\end{align*}
\end{lemma}

\noindent The mild formulation of the Kolmogorov equation is given by
\begin{align}
u_{t}=P_{T-t}u_{T}+\int_{t}^{T}P_{r-t}(V_{r}\cdot\nabla u_{r}-f_{r})dr=:P_{T-t}u_{T}+J^{T}(V\cdot\nabla u-f)(t).
\end{align} 
Due to the Schauder estimates, considering a singular terminal condition with $u_{T}\in\calC^{\beta +}_{p}$, we obtain that $\norm{P_{T-t}u_{T}}_{\calC^{\alpha+\beta}_{p}}$ blows up for $t\to T$ and the blow-up is of order $\gamma\in (0,1)$. 
This motivates the definition of blow-up spaces below, from which we can build the solution space in the next section.\\
\noindent For $\gamma\in (0,1)$, $T>0$ and $\overline{T}\in (0,T]$, and a Banach space $X$, let us define the blow-up space
\begin{align*}
\mathcal{M}_{\overline{T},T}^{\gamma}X:=\{u:[T-\overline{T},T)\to X\mid t\mapsto (T-t)^{\gamma}u_{t}\in C([T-\overline{T},T),X)\},
\end{align*} with $\norm{u}_{\mathcal{M}_{\overline{T},T}^{\gamma}X}:=\sup_{t\in [T-\overline{T},T)}(T-t)^{\gamma}\norm{u_t}_{X}$ and $\mathcal{M}_{\overline{T},T}^{0}X:=C([T-\overline{T},T),X)$. For $\overline{T}=T$, we use the notation $\mathcal{M}^{\gamma}_{T}X:=\mathcal{M}_{T,T}^{\gamma}X$. For $\vartheta\in (0,1]$, $\gamma\in (0,1)$, we furthermore define
\begin{align*}
C_{\overline{T},T}^{\gamma,\vartheta}X:=\biggl\{u:[T-\overline{T},T)\to X\biggm| \norm{f}_{C_{T}^{\gamma,\vartheta}X}:=\sup_{0\leqslant s<t< T}\frac{(T-t)^{\gamma}\norm{f_{t}-f_{s}}_{X}}{\abs{t-s}^{\vartheta}}<\infty\biggr\}
\end{align*} and $C_{T}^{\gamma,\vartheta}X:=C_{T,T}^{\gamma,\vartheta}X$. 
Let us also define for $\vartheta\in (0,1]$, $\overline{T}\in (0,T]$, the space of $\vartheta$-Hölder continuous functions on $[T-\overline{T},T]$ with values in $X$,
\begin{align*}
C_{\overline{T},T}^{\vartheta}X:=\biggl\{u:[T-\overline{T},T]\to X\biggm| \norm{u}_{C_{T}^{\vartheta}X}:=\sup_{T-\overline{T}\leqslant s<t\leqslant T}\frac{\norm{u_{t}-u_{s}}_{X}}{\abs{t-s}^{\vartheta}}<\infty\biggr\} 
\end{align*} and $C_{T}^{\vartheta}X:=C_{T,T}^{\vartheta}X$. We set $C_{\overline{T},T}^{0,\vartheta}X:=C^{\vartheta}([T-\overline{T},T),X)$.\\
We have the trivial estimates 
\begin{align}\label{eq:blow-up-inclusion}
\norm{u}_{\mathcal{M}_{\overline{T},T}^{\gamma_{1}}X}\leqslant\overline{T}^{\gamma_{1}-\gamma_{2}}\norm{u}_{\mathcal{M}_{\overline{T},T}^{\gamma_{2}}X},\quad \norm{u}_{C_{\overline{T},T}^{\gamma_{1},\vartheta_{1}}X}\leqslant\overline{T}^{(\gamma_{1}-\gamma_{2})+(\vartheta_{2}-\vartheta_{1})}\norm{u}_{ C_{\overline{T},T}^{\gamma_{2},\vartheta_{2}}X}
\end{align} for $0\leqslant\gamma_{2}\leqslant\gamma_{1}<1$ and $0<\vartheta_{1}\leqslant\vartheta_{2}\leqslant 1$. Moreover, we have that for a subinterval $[T-2\overline{T},T-\overline{T}]\subset [0,T]$ with $0<\overline{T}\leqslant \frac{T}{2}$,
\begin{align}\label{eq:subinterval-blow-up}
\norm{u}_{\mathcal{M}_{\overline{T},T-\overline{T}}^{0}X}\leqslant\overline{T}^{-\gamma}\norm{u}_{\mathcal{M}_{T}^{\gamma}X}.
\end{align}
\end{section}

\begin{section}{Schauder theory and commutator estimates for blow-up spaces}\label{section:singular-gen-eq}
In this section, we define the solution space $\mathcal{L}_{T}^{\gamma,\alpha+\beta}$ and prove Schauder and commutator estimates. We conclude the section with interpolation estimates for the solution spaces.\\ 

\noindent Heuristically, the solution space shall combine maximal space regularity (i.e. $\alpha+\beta$) in a time-blow-up space with maximal time regularity (i.e. Lipschitz) in a space of low space regularity. By interpolation, the solution will then also admit all time and space regularities ``in between".\\
Let us thus define for $\gamma\in (0,1)$ and $\theta\in \R$, $p\in[1,\infty]$, the space
\begin{align}\label{eq:function-space}
\mathcal{L}_{T}^{\gamma,\theta}:=\mathcal{M}_{T}^{\gamma}\calC^{\theta}_{p}\cap C_{T}^{1-\gamma}\calC^{\theta-\alpha}_{p}\cap C_{T}^{\gamma,1}\calC^{\theta-\alpha}_{p}.
\end{align} 
We moreover define for $\gamma=0$, 
\begin{align}\label{eq:gamma0}
\mathcal{L}_{T}^{0,\theta}:= C_{T}^{1}\calC^{\theta-\alpha}_{p}\cap C_{T}\calC^{\theta}_{p},
\end{align} where $C^{1}_{T}X$ denotes the space of $1$-Hölder or Lipschitz functions with values in $X$.\\ 
For $\overline{T}\in (0,T)$, we define $\mathcal{L}_{\overline{T},T}^{\gamma,\theta}:=\mathcal{M}_{\overline{T},T}^{\gamma}\calC^{\theta}_{p}\cap C_{\overline{T},T}^{1-\gamma}\calC^{\theta-\alpha}_{p}\cap C_{\overline{T},T}^{\gamma,1}\calC^{\theta-\alpha}_{p}$ and similarly $\mathcal{L}_{\overline{T},T}^{0,\theta}$.\\
The spaces $\mathcal{L}_{T}^{\gamma,\theta}$ are Banach spaces equipped with the norm
\begin{align*}
\norm{u}_{\mathcal{L}_{T}^{\gamma,\theta}}&:=\norm{u}_{\mathcal{M}_{T}^{\gamma}\calC^{\theta}_{p}}+\norm{u}_{C_{T}^{1-\gamma}\calC^{\theta-\alpha}_{p}} +\norm{u}_{C_{T}^{\gamma,1}\calC^{\theta-\alpha}_{p}}\\&=\squeeze[1]{\sup_{t\in[0,T)}(T-t)^{\gamma}\norm{u_{t}}_{\calC^{\theta}_{p}}+\sup_{0\leqslant s<t\leqslant T}\frac{\norm{u_{t}-u_{s}}_{\calC^{\theta-\alpha}_{p}}}{\abs{t-s}^{1-\gamma}}+\sup_{0\leqslant s<t< T}\frac{(T-t)^{\gamma}\norm{u_{t}-u_{s}}_{\calC^{\theta-\alpha}_{p}}}{\abs{t-s}}.}
\end{align*}
Notice, that $u\in\mathcal{L}_{T}^{\gamma,\theta}$ in particular implies that $t\mapsto\norm{u_{t}}_{\calC^{\theta-\alpha}}$ is $(1-\gamma)$-Hölder continuous at $t=T$.\\ 
\noindent The next corollary proves estimates for the semigroup $(P_{t})$ of $(-\La)$ acting on the spaces $\mathcal{L}_{T}^{\gamma,\theta}$. 
We will need the following auxillary lemma. In particular, the lemma can be applied, to show that the inverse fractional Laplacian improves space regularity by $\alpha$ (and not only by $\theta<\alpha$). It is a slight generalization of \cite[Lemma A.9, (A.1)]{Gubinelli2015Paracontrolled}. Its proof can be found in \cref{Appendix A}.
\begin{lemma}\label{lem:maxreg}
Let $\sigma\in\R$, $p\in [1,\infty]$, $\gamma\in[0,1)$, $\epsilon\in (0,1)$ and $\varsigma\geqslant 0$. Let moreover $f:\mathring{\Delta}_{T}\to\mathcal{S}'$, $\mathring{\Delta}_{T}:=\{(t,r)\in[0,T]^{2}\mid t<r\}$, be such that there exists $C>0$ such that for all $j\geqslant -1$ and $0\leqslant t< r\leqslant T$, for the Littlewood-Paley blocks holds
\begin{align*}
\norm{\Delta_{j}f_{t,r}}_{L^{p}}\leqslant C(T-r)^{-\gamma}\min(2^{-j\sigma},2^{-j(\sigma+\varsigma+\epsilon\varsigma)}(r-t)^{-(1+\epsilon)}).
\end{align*}
Then it follows that for all $t\in[0,T]$
\begin{align}
\norm[\bigg]{\int_{t}^{T}f_{t,r}dr}_{\calC^{\sigma+\varsigma}_{p}}\leqslant [2C \max(\epsilon^{-1},(1-\gamma)^{-1})](T-t)^{-\gamma}.
\end{align} 
\end{lemma} 
\begin{corollary}[Schauder estimates]\label{cor:schauder}
Let $(P_{t})$ and $\nu$ be as in \cref{schauder}. Let $T>0$, $\overline{T}\in (0,T]$. For $t\in[T-\overline{T},T]$ we define $J^{T}v (t)=J^{T}(v) (t):=\int_{t}^{T}P_{r-t}v(r)dr$. Then we have for $\beta\in\R$, $\vartheta\in[0,\alpha]$, $\gamma\in [\vartheta/\alpha,1]$,
\begin{align}\label{j1}
\norm{P_{T-\cdot}w}_{\mathcal{L}_{\overline{T},T}^{\gamma,\beta+\vartheta}}\lesssim \overline{T}^{(\gamma\alpha-\vartheta)/\alpha} \norm{w}_{\calC^{\beta}_{p}}
\end{align} 
and for $0\leqslant\gamma'\leqslant\gamma< 1$,
\begin{align}\label{j3}
\norm{J^{T}v}_{\mathcal{L}_{\overline{T},T}^{\gamma,\beta+\alpha}}\lesssim \overline{T}^{\gamma-\gamma'} \norm{v}_{\mathcal{M}_{\overline{T},T}^{\gamma'}\calC^{\beta}_{p}}.
\end{align} 
\end{corollary}
\begin{proof}
For \eqref{j1} we only prove the estimate in $C_{\overline{T},T}^{1-\gamma}\calC_{p}^{\beta+\vartheta-\alpha}$ and in $C_{\overline{T},T}^{\gamma,1}\calC_{p}^{\beta+\vartheta-\alpha}$, the estimate in $\mathcal{M}^{\gamma}_{\overline{T},T}\calC_{p}^{\beta+\vartheta}$ follows from a direct application of \cref{schauder}.\\
Therefore we write $P_{T-t}w-P_{T-s}w=P_{T-t}(\operatorname{Id}-P_{t-s})w$ for $T-\overline{T}\leqslant s<t\leqslant T$ and use \cref{schauder} to conclude
\begin{align*}
\norm{P_{T-t}w-P_{T-s}w}_{\calC_{p}^{\beta+\vartheta-\alpha}}\lesssim \norm{(\operatorname{Id}-P_{t-s})w}_{\calC_{p}^{\beta+\vartheta-\alpha}}&\lesssim (t-s)^{1-\vartheta/\alpha}\norm{w}_{\calC_{p}^{\beta}}\\&\lesssim \overline{T}^{(\gamma\alpha-\vartheta)/\alpha}(t-s)^{1-\gamma}\norm{w}_{\calC_{p}^{\beta}}
\end{align*} using $0\leqslant\vartheta\leqslant\alpha$ and $\gamma\geqslant\vartheta/\alpha$. This controls $\norm{P_{T-\cdot}w}_{C_{\overline{T},T}^{1-\gamma}\calC_{p}^{\beta+\gamma-\alpha}}$. To bound the norm $\norm{P_{T-\cdot}w}_{C_{\overline{T},T}^{\gamma, 1}\calC_{p}^{\beta+\gamma-\alpha}}$, we note that 
\begin{align*}
\norm{P_{T-t}w-P_{T-s}w}_{\calC_{p}^{\beta+\vartheta-\alpha}}&\lesssim (T-t)^{-\vartheta/\alpha}\norm{(\operatorname{Id}-P_{t-s})w}_{\calC_{p}^{\beta-\alpha}}\\&\lesssim (T-t)^{-\vartheta/\alpha}(t-s)\norm{w}_{\calC_{p}^{\beta}}\\&\lesssim \overline{T}^{(\gamma\alpha-\vartheta)/\alpha}(T-t)^{-\gamma}(t-s)\norm{w}_{\calC_{p}^{\beta}}.
\end{align*}
To estimate the $\mathcal{M}^{\gamma}_{\overline{T},T}\calC_{p}^{\beta+\alpha}$-norm in \eqref{j3}, we use \cref{lem:maxreg} with $f_{t,r}=P_{r-t}v_{r}$ and $\sigma=\beta$, $\varsigma=\alpha$, to obtain for $t\in[T-\overline{T},T]$ 
\begin{align*}
(T-t)^{\gamma}\norm{J^{T}v(t)}_{\calC_{p}^{\beta+\alpha}}=(T-t)^{\gamma-\gamma'}(T-t)^{\gamma'}\norm[\bigg]{\int_{t}^{T}P_{r-t}v_{r}dr}_{\calC_{p}^{\beta+\alpha}}
\lesssim \overline{T}^{\gamma-\gamma'}\norm{v}_{\mathcal{M}_{\overline{T},T}^{\gamma'}\calC_{p}^{\beta}}.
\end{align*} 
To prove the bounds on the time regularity in \eqref{j3} we write
\begin{align*}
J^{T}(v)_{t}-J^{T}(v)_{s}=\int_{s}^{t}P_{r-s}v_{r}dr-(P_{t-s}-\operatorname{Id})\paren[\bigg]{\int_{t}^{T}P_{r-t}v_{r}dr},
\end{align*} for $T-\overline{T}\leqslant s< t\leqslant T$. We can estimate by \cref{schauder}
\begin{align*}
\norm[\bigg]{\int_{s}^{t}P_{r-s}v_{r}dr}_{\calC^{\beta}_{p}}&\leqslant\int_{s}^{t}\norm{P_{r-s}v_{r}}_{\calC^{\beta}_{p}}dr\\&\lesssim \norm{v}_{\mathcal{M}^{\gamma'}_{\overline{T},T}\calC_{p}^{\beta}}\int_{s}^{t}\abs{T-r}^{-\gamma'}dr\\&\lesssim\norm{v}_{\mathcal{M}^{\gamma'}_{\overline{T},T}\calC_{p}^{\beta}}(\abs{T-s}^{1-\gamma'}-\abs{T-t}^{1-\gamma'})\\&\leqslant\overline{T}^{\gamma-\gamma'}\abs{t-s}^{1-\gamma}\norm{v}_{\mathcal{M}^{\gamma'}_{\overline{T},T}\calC_{p}^{\beta}},
\end{align*} using that $0\leqslant\gamma'\leqslant\gamma<1$ and the estimate 
\begin{align*}
\abs{T-t}^{1-\gamma'}-\abs{T-s}^{1-\gamma'}\leqslant \abs{t-s}^{1-\gamma'}\leqslant\overline{T}^{\gamma-\gamma'}\abs{t-s}^{1-\gamma}.
\end{align*} 
On the other hand, we can also estimate that term by
\begin{align*}
\norm[\bigg]{\int_{s}^{t}\!P_{r-s}v_{r}dr}_{\calC^{\beta}_{p}}&\lesssim \norm{v}_{\mathcal{M}^{\gamma'}_{\overline{T},T}\calC_{p}^{\beta}}\int_{s}^{t}\!\abs{T-r}^{-\gamma'}dr\\&\leqslant \norm{v}_{\mathcal{M}^{\gamma'}_{\overline{T},T}\calC_{p}^{\beta}}\abs{T-t}^{-\gamma'}\int_{s}^{t}\!dr
\\&\leqslant \overline{T}^{\gamma-\gamma'}\norm{v}_{\mathcal{M}^{\gamma'}_{\overline{T},T}\calC_{p}^{\beta}}\abs{T-t}^{-\gamma}\abs{t-s}.
\end{align*}
Moreover, by \cref{schauder} for $\vartheta=\alpha$ and \cref{lem:maxreg}, we obtain that
\begin{align*}
\norm[\bigg]{(P_{t-s}-\operatorname{Id})\paren[\bigg]{\int_{t}^{T}P_{r-t}v_{r}dr}}_{\calC^{\beta}_{p}}&\lesssim\abs{t-s}\norm[\bigg]{\int_{t}^{T}P_{r-t}v_{r}dr}_{\calC^{\beta+\alpha}_{p}}\\&\lesssim\abs{t-s}\norm{v}_{\mathcal{M}^{\gamma'}_{\overline{T},T}\calC_{p}^{\beta}}(T-t)^{-\gamma'}\\&\lesssim \abs{t-s}\overline{T}^{\gamma-\gamma'}\norm{v}_{\mathcal{M}^{\gamma'}_{\overline{T},T}\calC_{p}^{\beta}}(T-t)^{-\gamma},
\end{align*} and on the other hand we can estimate by \cref{schauder} for $\vartheta=(1-\gamma)\alpha$,
\begin{align*}
\MoveEqLeft
\norm[\bigg]{(P_{t-s}-\operatorname{Id})\paren[\bigg]{\int_{t}^{T}P_{r-t}v_{r}dr}}_{\calC^{\beta}_{p}}\\&\lesssim\abs{t-s}^{(\alpha-\gamma\alpha)/\alpha}\norm[\bigg]{\int_{t}^{T}P_{r-t}v_{r}dr}_{\calC^{\beta+\alpha-\gamma\alpha}_{p}}\\&\lesssim\abs{t-s}^{(\alpha-\gamma\alpha)/\alpha}\norm{v}_{\mathcal{M}^{\gamma'}_{\overline{T},T}\calC_{p}^{\beta}}\int_{t}^{T}(T-r)^{-\gamma'}(t-r)^{(\gamma\alpha-\alpha)/\alpha}dr\\&\lesssim\abs{t-s}^{1-\gamma}\norm{v}_{\mathcal{M}^{\gamma'}_{\overline{T},T}\calC_{p}^{\beta}}(T-t)^{\gamma-\gamma'}\\&\lesssim\abs{t-s}^{1-\gamma}\norm{v}_{\mathcal{M}^{\gamma'}_{\overline{T},T}\calC_{p}^{\beta}}\overline{T}^{\gamma-\gamma'},
\end{align*} where we used that $\gamma>0$ and that $\gamma'\leqslant\gamma <1$ (if $\gamma=0$, we can use the previous estimate instead).
\end{proof}

\begin{remark}
A less general approach for dealing with singular initial conditions in paracontrolled equations was developed in \cite{Gubinelli2017KPZ}. The function spaces above \eqref{eq:function-space} seem more flexible, and actually there is a mistake in the singular Schauder estimates in \cite[Lemma 6.6]{Gubinelli2017KPZ}: Equation (49) therein is only true for $\beta \in (0,2-\alpha)$, i.e. only for distributional initial conditions\footnote{We thank Ruhong Jin for pointing out this mistake.}, and $\beta \in (-\alpha,0)$ would force $u_0=0$.
\end{remark}

\noindent Next, we prove a commutator estimate for the $J^{T}$-operator and the paraproduct. 
\begin{lemma}[Commutator estimates]\label{lem:sharp}
Let $T>0$ and $\overline{T}\in (0,T]$ and let $\varsigma\in\R$, $\sigma\in (0,1)$ and $p\in[1,\infty]$. Let $\alpha\in (1,2]$ and $\gamma\in[0,1)$. Then for $u\in\calC^{\sigma}_{p}$, $v\in\calC^{\varsigma}$
the following semigroup commutator estimate holds
\begin{align}\label{eq:c1}
\norm{t\mapsto P_{T-t}( u\para v)- u\para P_{T-t}(v)}_{\mathcal{L}_{T}^{\gamma,\sigma+\varsigma+\gamma\alpha}}\lesssim \norm{u}_{\calC^{\sigma}_{p}}\norm{v}_{\calC^{\varsigma}}.
\end{align}
Furthermore, for  $g\in \mathcal{L}_{\overline{T},T}^{\gamma',\sigma}$ with $0\leqslant\gamma'\leqslant\gamma <1$ and $h\in \CTcalC^{\varsigma}$, we have
\begin{align}\label{eq:c3}
\norm{J^{T}( g\para h)- g\para J^{T}(h)}_{\mathcal{L}_{\overline{T},T}^{\gamma,\sigma+\varsigma+\alpha}}\lesssim \overline{T}^{\gamma-\gamma'}\norm{g}_{\mathcal{L}_{\overline{T},T}^{\gamma',\sigma}}\norm{h}_{C_{T}\mathcal{C}^{\varsigma}}.
\end{align} 
\end{lemma}
\begin{remark}
It was already known that the commutator for the $J^{T}$-operator from the lemma allows for more space regularity than both of its summands. The above commutator estimate moreover yields a gain in time regularity, i.e. $J^{T}(g\para h)-g\para J^{T}(h)\in C_{T}^{1-\gamma}\calC_{p}^{\sigma+\varsigma}\cap C_{T}^{\gamma,1}\calC_{p}^{\sigma+\varsigma}$, provided that $g\in\mathcal{L}^{\gamma,\sigma}_{T}$.
\end{remark}

\begin{proof}
Recall that $\mathcal{L}^{\gamma,\sigma+\varsigma+\gamma\alpha}_{T}$ is equipped with the sum of the norms in 
\begin{align*}
\mathcal{M}_{T}^{\gamma}\calC_{p}^{\sigma+\varsigma+\alpha\gamma},\quad C_{T}^{\gamma,1}\calC_{p}^{\sigma+\varsigma+\alpha\gamma-\alpha}\quad\text{ and }\quad  C_{T}^{1-\gamma}\calC_{p}^{\sigma+\varsigma+\alpha\gamma-\alpha},
\end{align*} that we need to estimate below.\\
For \eqref{eq:c1}, the estimate in $\mathcal{M}_{T}^{\gamma}\calC_{p}^{\sigma+\varsigma+\alpha\gamma}$ follows directly by the semigroup commutator \cref{schaudercom} applied to $\vartheta=\gamma\alpha$. For the estimate in $C_{T}^{\gamma,1}\calC_{p}^{\sigma+\varsigma+\alpha\gamma-\alpha}\cap C_{T}^{1-\gamma}\calC_{p}^{\sigma+\varsigma+\alpha\gamma-\alpha}$ we write for $0\leqslant s\leqslant t\leqslant T$,
\begin{align*}
\MoveEqLeft
P_{T-t}( u\para v)- u\para P_{T-t}(v)-(P_{T-s}( u\para v)- u\para P_{T-s}(v))\\&=(\operatorname{Id}-P_{t-s})[P_{T-t}(u\para v)-u\para P_{T-t}v]\\&\qquad+[u\para P_{t-s}P_{T-t}v-P_{t-s}(u\para P_{T-t}v)].
\end{align*} 
The first summand we can estimate by the semigroup estimates (\cref{schauder}) for $\operatorname{Id}-P_{t-s}$ and the commutator estimate in $\mathcal{M}_{T}^{\gamma}\calC_{p}^{\sigma+\varsigma+\alpha\gamma}$, obtaining
\begin{align*}
\MoveEqLeft
\norm{(\operatorname{Id}-P_{t-s})[P_{T-t}(u\para v)-u\para P_{T-t}v]}_{\calC^{\sigma+\varsigma+\alpha\gamma-\alpha}_{p}}\\&\lesssim\abs{t-s}\norm{[P_{T-t}(u\para v)-u\para P_{T-t}v]}_{\calC^{\sigma+\varsigma+\alpha\gamma}_{p}}
\\&\lesssim(T-t)^{-\gamma}\abs{t-s}\norm{u}_{\calC^{\sigma}_{p}}\norm{v}_{\calC^{\varsigma}}.
\end{align*} 
This gives the estimate in $C_{T}^{\gamma,1}\calC_{p}^{\sigma+\varsigma+\alpha(\gamma-1)}$.  Analogously we estimate the $C_{T}^{1-\gamma}\calC_{p}^{\sigma+\varsigma+\alpha(\gamma-1)}$-norm using the Schauder estimates for $\operatorname{Id}-P_{t-s}$ (obtaining a factor of $\abs{t-s}^{1-\gamma}$) and the commutator in $C_{\overline{T},T}\calC_{p}^{\sigma+\varsigma}$, i.e.
\begin{align*}
\MoveEqLeft
\norm{(\operatorname{Id}-P_{t-s})[P_{T-t}(u\para v)-u\para P_{T-t}v]}_{\calC^{\sigma+\varsigma+\alpha\gamma-\alpha}_{p}}\\&\lesssim\abs{t-s}^{1-\gamma}\norm{[P_{T-t}(u\para v)-u\para P_{T-t}v]}_{\calC^{\sigma+\varsigma}_{p}}
\\&\lesssim\abs{t-s}^{1-\gamma}\norm{u}_{\calC^{\sigma}_{p}}\norm{v}_{\calC^{\varsigma}}.
\end{align*}
The second summand can be estimated using the semigroup commutator (\cref{schaudercom}) for $\vartheta=(\gamma-1)\alpha\geqslant -\alpha$ and the semigroup estimate \eqref{eq:schauder2}, such that
\begin{align*}
\MoveEqLeft
\norm[\big]{P_{t-s}(u\para P_{T-t}v)-u\para P_{t-s}P_{T-t}v}_{\calC^{\sigma+\varsigma+\alpha\gamma-\alpha}_{p}}\\&\lesssim\abs{t-s}^{1-\gamma}\norm{u}_{\calC^{\sigma}_{p}}\norm{P_{T-t}v}_{\calC^\varsigma}\lesssim\abs{t-s}^{1-\gamma}\norm{u}_{\calC^{\sigma}_{p}}\norm{v}_{\calC^\varsigma}.
\end{align*} 
Using instead the semigroup commutator for $\vartheta=-\alpha\geqslant -\alpha$ and again the semigroup estimate \eqref{eq:schauder2} yields
\begin{align}\label{eq:id-p-commutator}
\MoveEqLeft
\norm[\big]{P_{t-s}(u\para P_{T-t}v)-u\para P_{t-s}P_{T-t}v}_{\calC^{\sigma+\varsigma+\alpha\gamma-\alpha}_{p}}\nonumber
\\&\lesssim\abs{t-s}\norm{u}_{\calC^{\sigma}_{p}}\norm{P_{T-t}v}_{\calC^{\varsigma+\alpha\gamma}}\nonumber
\\&\lesssim\abs{t-s}(T-t)^{-\gamma}\norm{u}_{\calC^{\sigma}_{p}}\norm{v}_{\calC^\varsigma}.
\end{align} 
Together, we obtain \eqref{eq:c1}.
For \eqref{eq:c3}, we first prove that $C(g,h):=J^{T}(g\para h)-g\para J^{T}(h)\in\mathcal{M}_{\overline{T},T}^{\gamma}\calC_{p}^{\sigma+\varsigma+\alpha}$.
To that end, we write
\begin{align*}
\squeeze[1]{C(g,h)_{t}=\!\int_{t}^{T}\!\paren[\big]{P_{r-t}(g_{r}\para h_{r})-g_{r}\para P_{r-t}h_{r}}dr+\int_{t}^{T}\!(g_{r}-g_{t})\para P_{r-t}h_{r}dr=:I_{1}(t)+I_{2}(t).}
\end{align*}
To estimate $I_{1}$, we utilize \cref{lem:maxreg} for $f_{t,r}=P_{r-t}(g_{r}\para h_{r})-g_{r}\para P_{r-t}h_{r}$,  
where the assumptions of the lemma are satisfied by the semigroup commutator estimate (\cref{schaudercom}). Then, we obtain
\begin{align*}
\norm{I_{1}(t)}_{\calC^{\sigma+\varsigma+\alpha}_{p}}
\lesssim\norm{g}_{\mathcal{M}^{\gamma'}_{\overline{T},T}\calC_{p}^{\sigma}}\norm{h}_{C_{T}\mathcal{C}^{\varsigma}}(T-t)^{-\gamma'}
\lesssim\overline{T}^{\gamma-\gamma'}\norm{g}_{\mathcal{M}^{\gamma'}_{\overline{T},T}\calC_{p}^{\sigma}}\norm{h}_{C_{T}\mathcal{C}^{\varsigma}}(T-t)^{-\gamma}.
\end{align*} 
For $I_{2}$, we apply \cref{lem:maxreg} for $f_{t,r}:=(g_{r}-g_{t})\para P_{r-t}h_{r}$. We check the assumptions on $f_{t,r}$ of that lemma, using the time regularity of $g$, as well as the paraproduct estimate (using $\sigma-\alpha<0$) and the semigroup estimates. Then, choosing $\theta=0$ or $\theta=(1+\epsilon)\alpha$ for $\epsilon\in[0,1]$, we estimate (the estimate is in fact valid for all $\theta\geqslant -\alpha$)
\begin{align*}
\norm{(g_{r}-g_{t})\para P_{r-t}h_{r}}_{\calC^{\sigma+\varsigma+\theta}_{p}}&=\norm{(g_{r}-g_{t})\para P_{r-t}h_{r}}_{\calC^{(\sigma-\alpha)+(\varsigma+\theta+\alpha)}_{p}}\\&\lesssim (T-r)^{-\gamma'}(r-t)^{-\theta/\alpha}\norm{h}_{C_{T}\calC^{\varsigma}}\norm{g}_{C_{\overline{T},T}^{\gamma',1}\calC^{\sigma-\alpha}_{p}}
\end{align*} 
Applying \cref{lem:maxreg} yields then the estimate for $I_{2}$:
\begin{align*}
\norm{I_{2}(t)}_{\calC^{\sigma+\varsigma+\alpha}_{p}}
\lesssim\overline{T}^{\gamma-\gamma'}\norm{g}_{C_{\overline{T},T}^{\gamma',1}\calC^{\sigma-\alpha}_{p}}\norm{h}_{C_{T}\mathcal{C}^{\varsigma}_{\R^{d}}}(T-t)^{-\gamma}.
\end{align*} 
Next, we prove the time regularity estimates on the commutator $C(g,h)$. For that, we write for $T-\overline{T}\leqslant s\leqslant t\leqslant T$, 
\begin{align*}
\MoveEqLeft
J^{T}( g\para h)_t- g_t\para J^{T}(h)_t-(J^{T}( g\para h)_s- g_s\para J^{T}(h)_s)\\&=-\int_{s}^{t}P_{r-s}(g_{r}\para h_{r})dr-(P_{t-s}-\operatorname{Id})\paren[\bigg]{\int_{t}^{T}P_{r-t}(g_{r}\para h_{r})dr}\\&\qquad+g_{s}\para\int_{s}^{t}P_{r-s}h_{r}dr-g_{s,t}\para\int_{t}^{T}P_{r-t}h_{r}dr+g_{s}\para (P_{t-s}-\operatorname{Id})\paren[\bigg]{\int_{t}^{T}P_{r-t}h_{r}dr}\\&=A_{st}+B_{st}+C_{st},
\end{align*} where we define
\begin{align*}
A_{st}:=g_{s}\para\int_{s}^{t}P_{r-s}h_{r}dr-\int_{s}^{t}P_{r-s}(g_{r}\para h_{r})dr
\end{align*} and
\begin{align*}
B_{st}:=-g_{s,t}\para\int_{t}^{T}P_{r-t}h_{r}dr,
\end{align*} where $g_{s,t}:=g_{t}-g_{s}$ and
\begin{align*}
C_{st}:=g_{s}\para P_{t-s}\paren[\bigg]{\int_{t}^{T}P_{r-t}h_{r}dr}-P_{t-s}\paren[\bigg]{\int_{t}^{T}P_{r-t}(g_{r}\para h_{r})dr}.
\end{align*}
We will consider the terms $A_{st},B_{st}$ and $C_{st}$ separately and estimate each term in the $C_{\overline{T},T}^{1-\gamma}\calC^{\sigma+\varsigma}$-norm and in the $C_{\overline{T},T}^{\gamma,1}\calC^{\sigma+\varsigma}$-norm.\\
We start with $B_{st}$, using the time regularity of $g$, obtaining on the one hand
\begin{align*}
\norm{B_{st}}_{\calC^{\sigma+\varsigma}_{p}}&=\norm[\bigg]{g_{s,t}\para\int_{t}^{T}P_{r-t}h_{r}dr}_{\calC^{(\sigma-\alpha)+(\alpha+\varsigma)}_{p}}\\&\lesssim\norm{g}_{C_{\overline{T},T}^{1-\gamma'}\calC^{\sigma-\alpha}_{p}}\abs{t-s}^{1-\gamma'}\norm[\bigg]{\int_{t}^{T}P_{r-t}h_{r}dr}_{\calC^{\alpha+\varsigma}}\\&\lesssim\abs{t-s}^{1-\gamma'}\norm{g}_{C_{\overline{T},T}^{1-\gamma'}\calC^{\sigma-\alpha}_p}\norm{h}_{C_{T}\calC^{\varsigma}}
\\&\lesssim\overline{T}^{\gamma-\gamma'}\abs{t-s}^{1-\gamma}\norm{g}_{C_{\overline{T},T}^{1-\gamma'}\calC^{\sigma-\alpha}_p}\norm{h}_{C_{T}\calC^{\varsigma}},
\end{align*} using $\sigma-\alpha<0$ and \cref{lem:maxreg} for $f_{t,r}=P_{r-t}h_{r}$ to bound the time integral. On the other hand, along the same lines, using instead $g\in C_{\overline{T},T}^{\gamma',1}\calC^{\sigma-\alpha}_p$, we can estimate $B_{st}$ by
\begin{align*}
\norm{B_{st}}_{\calC^{\sigma+\varsigma}_p}\lesssim\overline{T}^{\gamma-\gamma'}\norm{h}_{C_{T}\calC^{\varsigma}}\norm{g}_{C_{\overline{T},T}^{\gamma',1}\calC^{\sigma-\alpha}_p}\abs{t-s}(T-t)^{-\gamma}.
\end{align*}
For $A_{st}$, we use the semigroup commutator (\cref{schaudercom}) for $\vartheta=0$, as well as the time regularity of $g$, which yields 
\begin{align*}
\MoveEqLeft
\norm{A_{st}}_{\calC^{\sigma+\varsigma}_p}\\&=\norm[\bigg]{\int_{s}^{t}P_{r-s}(g_{r}\para h_{r})dr-g_{s}\para\int_{s}^{t}P_{r-s}h_{r}dr}_{\calC^{\sigma+\varsigma}_p}\\&\leqslant\squeeze[1]{\norm[\bigg]{\int_{s}^{t}(P_{r-s}(g_{r}\para h_{r})-g_{r}\para P_{r-s}h_{r})dr}_{\calC^{\sigma+\varsigma}_p}\!\!\!\!+\norm[\bigg]{\int_{s}^{t}(g_{r}-g_{s})\para P_{r-s}h_{r}dr}_{\calC^{(\sigma-\alpha)+(\varsigma+\alpha)}_p}}\\&\leqslant\norm{g}_{\mathcal{M}^{\gamma'}_{\overline{T},T}\calC^{\sigma}}\norm{h}_{C_{T}\calC^{\varsigma}_p}\int_{s}^{t}(T-r)^{-\gamma'}dr+\norm{h}_{C_{T}\calC^{\varsigma}}\norm{g}_{C^{1-\gamma'}_{\overline{T},T}\calC^{\sigma-\alpha}_p}\int_{s}^{t}\abs{r-s}^{-\gamma'}dr\\&\lesssim\norm{g}_{\mathcal{L}_{\overline{T},T}^{\gamma',\sigma}}\norm{h}_{C_{T}\mathcal{C}^{\varsigma}}\paren[\big]{(T-s)^{1-\gamma'}-(T-t)^{1-\gamma'}+\abs{t-s}^{1-\gamma'}}\\&\lesssim\norm{g}_{\mathcal{L}_{\overline{T},T}^{\gamma',\sigma}}\norm{h}_{C_{T}\mathcal{C}^{\varsigma}}\overline{T}^{\gamma-\gamma'}\abs{t-s}^{1-\gamma}.
\end{align*} 
We can also estimate the term $A_{st}$ by
\begin{align*}
\norm{A_{st}}_{\calC^{\sigma+\varsigma}_p}&\leqslant\squeeze[1]{\norm{g}_{\mathcal{M}^{\gamma'}_{\overline{T},T}\calC^{\sigma}_p}\norm{h}_{C_{T}\calC^{\varsigma}}\int_{s}^{t}(T-r)^{-\gamma'}dr+\norm{h}_{C_{T}\calC^{\varsigma}}\norm{g}_{C^{\gamma',1}_{\overline{T},T}\calC^{\sigma-\alpha}_p}\int_{s}^{t}(T-r)^{-\gamma'}dr}\\&\lesssim (T-t)^{-\gamma'}\abs{t-s}\norm{h}_{C_{T}\calC^{\varsigma}}\paren[\bigg]{\norm{g}_{\mathcal{M}^{\gamma'}_{\overline{T},T}\calC^{\sigma}_p}+\norm{g}_{C^{\gamma',1}_{\overline{T},T}\calC^{\sigma-\alpha}_p}}\\&\lesssim \overline{T}^{\gamma-\gamma'}(T-t)^{-\gamma}\abs{t-s}\norm{h}_{C_{T}\calC^{\varsigma}}\norm{g}_{\mathcal{L}^{\gamma',\sigma}_{\overline{T},T}},
\end{align*} using $(T-r)^{-\gamma}\leqslant(T-t)^{-\gamma}$ for $r\in [s,t]$.
It is left to estimate the term $C_{st}$, that we first rewrite:
\begin{align}
C_{st}&=P_{t-s}\paren[\bigg]{\int_{t}^{T}P_{r-t}(g_{r}\para h_{r})dr}-g_{s}\para P_{t-s}\paren[\bigg]{\int_{t}^{T}P_{r-t}h_{r}dr}\nonumber\\&=(P_{t-s}-\operatorname{Id})\paren[\bigg]{\int_{t}^{T}P_{r-t}(g_{r}\para h_{r})dr-g_{s}\para\int_{t}^{T}P_{r-t}h_{r}dr}\label{eq:ct1}\\&\quad+P_{t-s}\paren[\bigg]{g_{s}\para\int_{t}^{T}P_{r-t}h_{r}dr}-g_{s}\para P_{t-s}\paren[\bigg]{\int_{t}^{T}P_{r-t}h_{r}dr}.\label{eq:ct2}
\end{align} 
To estimate the term in line \eqref{eq:ct1}, we use \cref{schauder} and the estimate for $I_{1}(t)+I_{2}(t)$ from above to obtain
\begin{align*}
\MoveEqLeft
\norm[\bigg]{(P_{t-s}-\operatorname{Id})\paren[\bigg]{\int_{t}^{T}(P_{r-t}(g_{r}\para h_{r}) -g_{s}\para P_{r-t}h_{r})dr}}_{\calC^{\varsigma+\sigma+\alpha-\alpha}_p}\\&\lesssim\abs{t-s}\norm[\bigg]{\int_{t}^{T}(P_{r-t}(g_{r}\para h_{r}) -g_{s}\para P_{r-t}h_{r})dr}_{\calC^{\varsigma+\sigma+\alpha}_p}\\&\lesssim\abs{t-s}(T-t)^{-\gamma}\overline{T}^{\gamma-\gamma'}\norm{g}_{\mathcal{L}_{\overline{T},T}^{\gamma',\sigma}}\norm{h}_{C_{T}\mathcal{C}^{\varsigma}}.
\end{align*} 
The term in line \eqref{eq:ct1}, we can also estimate differently using \cref{schauder} and an easier estimate for $I_{1}(t),I_{2}(t)$ using the semigroup estimates and $\alpha(1-\gamma')<\alpha$ to obtain
\begin{align*}
\MoveEqLeft
\norm[\bigg]{(P_{t-s}-\operatorname{Id})\paren[\bigg]{\int_{t}^{T}(P_{r-t}(g_{r}\para h_{r}) -g_{s}\para P_{r-t}h_{r})dr}}_{\calC^{\varsigma+\sigma}_p}\\&\lesssim\abs{t-s}^{1-\gamma'}\norm[\bigg]{\int_{t}^{T}(P_{r-t}(g_{r}\para h_{r}) -g_{s}\para P_{r-t}h_{r})dr}_{\calC^{\varsigma+\sigma+\alpha(1-\gamma')}_p}\\&\lesssim\abs{t-s}^{1-\gamma'}(\norm{I_{1}(t)}_{\calC^{\varsigma+\sigma+\alpha(1-\gamma')}_p}+\norm{I_{2}(t)}_{\calC^{\varsigma+\sigma+\alpha(1-\gamma')}_p})\\&\lesssim\abs{t-s}^{1-\gamma'}\norm{h}_{C_{T}\mathcal{C}^{\varsigma}}\paren[\bigg]{[\norm{g}_{\mathcal{M}^{\gamma'}_{\overline{T},T}\calC^{\sigma}_p}+\norm{g}_{C^{\gamma',1}_{\overline{T},T}\calC^{\sigma-\alpha}_p}]\int_{t}^{T}(T-r)^{-\gamma'}(r-t)^{-1+\gamma'}dr}\\&\lesssim\overline{T}^{\gamma-\gamma'}\abs{t-s}^{1-\gamma}\norm{h}_{C_{T}\mathcal{C}^{\varsigma}}\norm{g}_{\mathcal{L}^{\gamma',\sigma}_{\overline{T},T}}\int_{0}^{1}(1-r)^{-\gamma'}r^{-1+\gamma'}dr\\&\lesssim\overline{T}^{\gamma-\gamma'}\abs{t-s}^{1-\gamma}\norm{h}_{C_{T}\mathcal{C}^{\varsigma}}\norm{g}_{\mathcal{L}^{\gamma',\sigma}_{\overline{T},T}}.
\end{align*} 
To estimate the term in line \eqref{eq:ct2}, we use the commutator for $P_{t-s}$ for $\vartheta=-\alpha$ and again \cref{lem:maxreg} for $f_{t,r}=P_{r-t}h_{r}$, yielding 
\begin{align*}
\MoveEqLeft
\norm[\bigg]{P_{t-s}\paren[\bigg]{g_{s}\para\int_{t}^{T}P_{r-t}h_{r}dr}-g_{s}\para P_{t-s}\paren[\bigg]{\int_{t}^{T}P_{r-t}h_{r}dr)}}_{\calC^{\sigma+\varsigma+\alpha-\alpha}_p}\\&\lesssim\abs{t-s}\norm{h}_{C_{T}\mathcal{C}^{\varsigma}}\norm{g}_{\mathcal{M}^{\gamma'}_{\overline{T},T}\calC^{\sigma}_p}(T-s)^{-\gamma'}\norm[\bigg]{\int_{t}^{T}P_{r-t}h_{r}dr}_{\calC^{\varsigma+\alpha}}\\&\lesssim\overline{T}^{\gamma-\gamma'}\abs{t-s}\norm{h}_{C_{T}\mathcal{C}^{\varsigma}}\norm{g}_{\mathcal{M}^{\gamma'}_{\overline{T},T}\calC^{\sigma}_p}(T-s)^{-\gamma}.
\end{align*} 
Applying instead the semigroup commutator for $\vartheta=-(1-\gamma')\alpha$ yields
\begin{align*}
\MoveEqLeft
\norm[\bigg]{P_{t-s}\paren[\bigg]{g_{s}\para\int_{t}^{T}P_{r-t}h_{r}dr}-g_{s}\para P_{t-s}\paren[\bigg]{\int_{t}^{T}P_{r-t}h_{r}dr)}}_{\calC^{\sigma+\varsigma}_p}\\&\lesssim\abs{t-s}^{1-\gamma'}\norm{g}_{\mathcal{M}^{\gamma'}_{\overline{T},T}\calC^{\sigma}_p}(T-s)^{-\gamma'}\norm[\bigg]{\int_{t}^{T}P_{r-t}h_{r}dr}_{\calC^{\varsigma+\alpha(1-\gamma')}}\\&\lesssim\abs{t-s}^{1-\gamma'}\norm{h}_{C_{T}\mathcal{C}^{\varsigma}}\norm{g}_{\mathcal{M}^{\gamma'}_{\overline{T},T}\calC^{\sigma}_p}(T-s)^{-\gamma'}(T-t)^{\gamma'}\\&\lesssim\overline{T}^{\gamma-\gamma'}\abs{t-s}^{1-\gamma}\norm{h}_{C_{T}\mathcal{C}^{\varsigma}}\norm{g}_{\mathcal{M}^{\gamma'}_{\overline{T},T}\calC^{\sigma}_p},
\end{align*} 
where to bound the time integral, we used that $\alpha(1-\gamma')<\alpha$ and $s\leqslant t$. 
Together we obtain the desired estimates for $C_{st}$, which yield together with the estimates for $A_{st}$ and $B_{st}$ the claim.
\end{proof}
\begin{remark}
The proof of the commutator estimate does not apply if we consider instead of $g\in \mathcal{L}_{T}^{\gamma,\sigma}$, a function $g\in \mathcal{M}_{T}^{\gamma}\calC^{\sigma}\cap C_{T}^{1-\gamma}\calC^{\sigma-\alpha}$. The reason is the estimate for the term $I_{2}$ in the above proof, for which we need to employ that $g\in C_{T}^{\gamma,1}\calC^{\sigma-\alpha}$.  
\end{remark}
\noindent We conclude this section with interpolation estimates for the spaces $\mathcal{L}_{T}^{\gamma,\theta}$.
\begin{lemma}[Interpolation estimates]\label{lem:int-est}
Let $\gamma\in[0,1)$, $\theta\in [0,\alpha]$, $p\in[1,\infty]$. Let moreover $v\in\mathcal{L}_{T}^{\gamma,\theta}$. Then the following estimates hold true:\\ 
It follows that for $\theta\in (0,\alpha)$,
\begin{align}\label{eq:i2}
\norm{v}_{C_{T}^{\theta/\alpha}L^{p}}\lesssim \norm{v}_{\mathcal{L}_{T}^{0,\theta}}.
\end{align} 
Furthermore, for $\tilde{\theta}\in [0,\alpha]$, it holds that
\begin{align}\label{eq:i3a}
\norm{v}_{C_{T}^{\gamma,\tilde{\theta}/\alpha}\calC^{\theta-\tilde{\theta}}_p}\lesssim\norm{v}_{\mathcal{L}_{T}^{\gamma,\theta}}
\end{align} and
\begin{align}\label{eq:i3}
\norm{v}_{\mathcal{M}_{T}^{\gamma(1-\tilde{\theta}/\alpha)}\calC^{\theta-\tilde{\theta}}_{p}}\lesssim \norm{v}_{\mathcal{L}_{T}^{\gamma,\theta}}.
\end{align} 
If $v_{T}\in\calC^{\theta-\tilde{\theta}}_{p}$ and $\tilde{\theta}\in[\alpha\gamma,\alpha]$, then the following estimate holds true
\begin{align}\label{eq:i3b}
\norm{v_{t}}_{\calC^{\theta-\tilde{\theta}}_{p}}\lesssim (T-t)^{\tilde{\theta}/\alpha-\gamma}\norm{v}_{\mathcal{L}_{T}^{\gamma,\theta}}+\norm{v_{T}}_{\calC^{\theta-\tilde{\theta}}_{p}}.
\end{align}
\end{lemma}
\begin{remark}
For $\gamma=0$, $\theta\in (0,1]$ and a Banach space $X$, we recall that $C_{T}^{0,\theta}X=C_{T}^{\theta}X$.
\end{remark}

\begin{proof}
To prove \eqref{eq:i2} we let $0\leqslant s\leqslant t\leqslant T$ and estimate
\begin{align*}
\norm{v_{t}-v_{s}}_{L^{p}}&\leqslant\sum_{j}\norm{\Delta_{j}(v_{t}-v_{s})}_{L^{p}}\\&\lesssim\sum_{j:\,2^{-j}\leqslant\abs{t-s}^{1/\alpha}}2^{-j\theta}\norm{v}_{C_{T}\calC_{p}^{\theta}}+\sum_{j:\,2^{-j}>\abs{t-s}^{1/\alpha}}2^{-j(\theta-\alpha)}\abs{t-s}\norm{v}_{C_{T}^{1}\calC_{p}^{\theta-\alpha}}\\&\lesssim \abs{t-s}^{\theta/\alpha}\norm{v}_{C_{T}\calC_{p}^{\theta}}+\abs{t-s}^{\theta/\alpha}\norm{v}_{C_{T}^{1}\calC_{p}^{\theta-\alpha}},
\end{align*} using that $\theta> 0$ for the convergence of the geometric sum and that $\theta <\alpha$.
To prove \eqref{eq:i3a} and \eqref{eq:i3}, we let $\tilde{\theta}\in [0,\alpha]$.
Then we estimate for $s<t$,
\begin{align*}
\norm{\Delta_{j}(v_{t}-v_{s})}_{L^{p}}\lesssim (T-t)^{-\gamma}\min\Big(2^{-j\theta}\norm{v}_{\mathcal{M}_{T}^{\gamma}\calC_{p}^{\theta}},2^{-j(\theta-\alpha)}\abs{t-s}\norm{v}_{C_{T}^{\gamma,1}\calC_{p}^{\theta-\alpha}}\Big)
\end{align*} and for $t\in[0,T)$, 
\begin{align*}
\norm{\Delta_{j}v_{t}}_{L^{p}}\lesssim\min(2^{-j\theta}(T-t)^{-\gamma}\norm{v}_{\mathcal{M}_{T}^{\gamma}\calC_{p}^{\theta}},2^{-j(\theta-\alpha)}\norm{v}_{C_{T}\calC_{p}^{\theta-\alpha}}).
\end{align*} 
Thus by interpolation (that is, $\min(a,b)\leqslant a^{\epsilon}b^{1-\epsilon}$ for $a,b\geqslant 0$, $\epsilon\in [0,1]$) and using that $\norm{v}_{C_{T}\calC_{p}^{\theta-\alpha}}\lesssim\norm{v}_{\mathcal{L}_{T}^{\gamma,\theta}}$, we obtain
\begin{align*}
\norm{\Delta_{j}(v_{t}-v_{s})}_{L^{p}}&\lesssim (T-t)^{-\gamma}2^{-j\theta(1-\tilde{\theta}/\alpha)}2^{-j(\theta-\alpha)\tilde{\theta}/\alpha}\abs{t-s}^{\tilde{\theta}/\alpha}\norm{v}_{\mathcal{L}_{T}^{\gamma,\theta}}\\&=(T-t)^{-\gamma} 2^{-j(\theta-\tilde{\theta})}\abs{t-s}^{\tilde{\theta}/\alpha}\norm{v}_{\mathcal{L}_{T}^{\gamma,\theta}},
\end{align*} from which \eqref{eq:i3a} follows, and
\begin{align*}
\norm{\Delta_{j}v_{t}}_{L^{p}}&\lesssim 2^{-j\theta(1-\tilde{\theta}/\alpha)}(T-t)^{-\gamma(1-\tilde{\theta}/\alpha)}\norm{v}_{\mathcal{M}_{T}^{\gamma}\calC_{p}^{\theta}}^{1-\tilde{\theta}/\alpha}\,\, 2^{-j(\theta-\alpha)\tilde{\theta}/\alpha}\norm{v}_{C_{T}\calC_{p}^{\theta-\alpha}}^{\tilde{\theta}/\alpha}\\&\leqslant 2^{-j(\theta-\tilde{\theta})}\norm{v}_{\mathcal{L}_{T}^{\gamma,\theta}}(T-t)^{-\gamma(1-\tilde{\theta}/\alpha)},
\end{align*} which yields \eqref{eq:i3}.
Finally, if $(T-t)\geqslant 1$, then \eqref{eq:i3b} follows from \eqref{eq:i3} as
\begin{align*}
\norm{v_{t}}_{\calC^{\theta-\tilde{\theta}}_{p}}\lesssim \norm{v}_{\mathcal{L}_{T}^{\gamma,\theta}}(T-t)^{-\gamma(1-\tilde{\theta}/\alpha)}&\leqslant \norm{v}_{\mathcal{L}_{T}^{\gamma,\theta}}\\&\leqslant (T-t)^{\tilde{\theta}/\alpha-\gamma} [\norm{v}_{\mathcal{L}_{T}^{\gamma,\theta}}+\norm{v_{T}}_{\calC^{\theta-\tilde{\theta}}_{p}}^{1-\tilde{\theta}/\alpha}]+\norm{v_{T}}_{\calC^{\theta-\tilde{\theta}}_{p}}
\end{align*} using that $\tilde{\theta}/\alpha\geqslant \gamma$. If $(T-t)\leqslant 1$ , then \eqref{eq:i3b} follows from
\begin{align*}
\norm{v_{t}}_{\calC^{\theta-\tilde{\theta}}_{p}}\leqslant\norm{v_{t}-v_{T}}_{\calC^{\theta-\tilde{\theta}}_{p}}+\norm{v_{T}}_{\calC^{\theta-\tilde{\theta}}_{p}}
\end{align*} and
\begin{align*}
\MoveEqLeft
\norm{\Delta_{j}(v_{t}-v_{T})}_{L^{p}}\\&\lesssim \min\Big(2^{-j\theta}(T-t)^{-\gamma}\norm{v}_{\mathcal{M}_{T}^{\gamma}\calC_{p}^{\theta}}+2^{-j(\theta-\tilde{\theta})}\norm{v_{T}}_{\calC^{\theta-\tilde{\theta}}_{p}}, 2^{-j(\theta-\alpha)}(T-t)^{1-\gamma}\norm{v}_{C_{T}^{1-\gamma}\calC_{p}^{\theta-\alpha}}\Big)
\\&\leqslant \min\Big(2^{-j\theta}(T-t)^{-\gamma}\norm{v}_{\mathcal{M}_{T}^{\gamma}\calC_{p}^{\theta}}, 2^{-j(\theta-\alpha)}(T-t)^{1-\gamma}\norm{v}_{C_{T}^{1-\gamma}\calC_{p}^{\theta-\alpha}}\Big)+2^{-j(\theta-\tilde{\theta})}\norm{v_{T}}_{\calC^{\theta-\tilde{\theta}}_{p}}.
\end{align*} 
By interpolation as above, we thus have
\begin{align*}
\norm{\Delta_{j}(v_{t}-v_{T})}_{L^{p}}\lesssim 2^{-j(\theta-\tilde{\theta})}(T-t)^{\tilde{\theta}/\alpha-\gamma}\norm{v}_{\mathcal{L}_{T}^{\gamma,\theta}},
\end{align*} such that together \eqref{eq:i3b} follows.
\end{proof}
\end{section}

\begin{section}{Solving the Kolmogorov backward equation}\label{sec:singularKeq}
In this section, we develop a concise solution theory that simultaneously treats singular and non-singular terminal condition for the Kolmogorov backward equation.

\noindent We start by solving the Kolmogorov equation in the Young regime, that is $\beta>(1-\alpha)/2$. 
\begin{theorem}\label{thm:ygeneq}\footnote{The theorem is a generalization of \cite[Theorem 3.1]{kp} to regularity $\theta=\alpha+\beta$ and integrability $p\in[1,\infty]$.}
Let $\alpha\in (1,2]$, $\beta\in(\frac{1-\alpha}{2}, 0)$ and $p\in[1,\infty]$. Let $V\in\CTcalC^{\beta}_{\R^{d}}$, $f\in\CTcalC^{\beta}_{p}$ and $u^{T}\in \calC^{\alpha+\beta}_{p}$. Then the PDE 
\begin{align}\label{eq:pdem}
\partial_{t}u=\La u-V\cdot \nabla u+f,\quad u(T,\cdot)=u^{T},
\end{align}
has a unique mild solution $u\in C_{T}\calC^{\alpha+\beta}\cap C_{T}^{1}\calC^{\beta}$ (i.p.~by \eqref{eq:i2}, $u\in C_{T}^{(\alpha+\beta)/\alpha}L^{p}$). Moreover, the solution map 
\[
	\calC^{\alpha+\beta}_{p} \times \CTcalC^{\beta}_{p} \times \CTcalC^{\beta}_{\R^{d}} \ni (u^{T},f,V)\mapsto u \in \mathcal{L}_{T}^{0,\alpha+\beta}
\]
is continuous.\\
Furthermore, for a singular terminal conditions $u^{T}\in\calC^{(1-\gamma)\alpha+\beta}_{p}$ for $\gamma\in [0,1)$, the solution $u$ is obtained in $\mathcal{L}_{T}^{\gamma,\alpha+\beta}$.
\end{theorem}

\begin{proof}
Let $u^{T}\in\calC^{\alpha+\beta}_{p}$. We first prove, that the solution exists in $\mathcal{L}_{T}^{\gamma,\alpha+\beta}$ for any $\gamma\in(0,1)$. Afterwards we argue that indeed $u\in\mathcal{L}_{T}^{0,\alpha+\beta}$.
The proof follows from the Banach fixed point theorem applied to the map
\begin{align*}
\mathcal{L}_{\overline{T},T}^{\gamma,\alpha+\beta}\ni u\mapsto \Phi^{\overline T, T} (u)\in \mathcal{L}_{\overline{T},T}^{\gamma,\alpha+\beta}\text{ with }  \Phi^{\overline T, T}u(t)=P_{T-t}u^{T}+J^{T}(\nabla u\cdot V-f)(t),
\end{align*} where $J^{T}(v)(t)=\int_{t}^{T}P_{r-t}v(r)dr$. We show below, that for $\overline{T}\in (0,T]$ small enough, the map is a contraction. By the Schauder estimates (\cref{cor:schauder}), we obtain that $t\mapsto P_{T-t}u^{T}\in\mathcal{L}_{\overline{T},T}^{0,\alpha+\beta}$ and $J^{T}(f)\in\mathcal{L}_{\overline{T},T}^{0,\alpha+\beta}$. Furthermore, the Schauder estimates (\cref{cor:schauder}) and the interpolation estimate \eqref{eq:i3} from \cref{lem:int-est} yield that for $\gamma'\in (0,\gamma)$ chosen, such that $\gamma=\gamma'(1-\theta/\alpha)$ for a $\theta\in (0,\alpha+2\beta-1)$,
\begin{align*}
\norm{J^{T}(\nabla u\cdot V)}_{\mathcal{L}_{\overline{T},T}^{\gamma,\alpha+\beta}}&\lesssim \overline{T}^{\gamma-\gamma'}\norm{\nabla u\cdot V}_{\mathcal{M}_{\overline{T},T}^{\gamma '}\calC^{\beta}_{p}}\\&\lesssim \overline{T}^{\gamma-\gamma'}\norm{\nabla u}_{\mathcal{M}_{\overline{T},T}^{\gamma '}(\calC^{\alpha+\beta-1-\theta}_{p})^{d}}\norm{V}_{C_{T}\calC^{\beta}_{\R^{d}}}\\&\lesssim\overline{T}^{\gamma-\gamma'}\norm{u}_{\mathcal{L}_{\overline{T},T}^{\gamma ,\alpha+\beta}}\norm{V}_{C_{T}\calC^{\beta}_{\R^{d}}}.
\end{align*} 
Notice that due to the choice of $\theta$ the regularity of the resonant product $\nabla u\reso V$ is strictly positive.
Thus, for $\overline{T}\in (0,T]$ sufficiently small, 
$\Phi^{\overline T, T}$ is a contraction on $\mathcal{L}_{\overline{T},T}^{\gamma,\alpha+\beta}$ and we obtain a solution $u\in\mathcal{L}_{\overline{T},T}^{\gamma,\alpha+\beta}$ (i.e. the fixed point of the map).\\
By plugging the solution back in the contraction map and using the interpolation estimate \eqref{eq:i3b} for $\theta=\alpha+\beta, \tilde{\theta}=\gamma\alpha$ and $\gamma\in (0,(\alpha+2\beta-1)/\alpha)$, we then obtain
\begin{align}\label{eq:more-reg}
\norm{u}_{\mathcal{L}_{\overline{T},T}^{0,\alpha+\beta}}&=\norm{\Phi^{\overline{T},T}(u)}_{\mathcal{L}_{\overline{T},T}^{0,\alpha+\beta}}\nonumber\\&\lesssim\norm{P_{T-\cdot}u^{T}+J^{T}(f)}_{\mathcal{L}_{\overline{T},T}^{0,\alpha+\beta}}+\norm{u}_{C_{\overline{T},T}\calC^{\alpha+\beta-\gamma\alpha}}\norm{V}_{C_{T}\calC^{\beta}_{\R^{d}}}\nonumber\\&\lesssim\norm{P_{T-\cdot}u^{T}+J^{T}(f)}_{\mathcal{L}_{T}^{0,\alpha+\beta}}+[\norm{u}_{\mathcal{L}_{\overline{T},T}^{\gamma,\alpha+\beta}}+\norm{u_{T}}_{\calC^{\alpha+\beta}_{p}}]\norm{V}_{C_{T}\calC^{\beta}_{\R^{d}}}.
\end{align}
This implies that indeed $u\in \mathcal{L}^{0,\alpha+\beta}_{\overline{T},T}$ and we constructed the solution on $[T-\overline{T},T]$.\\
Moreover, the choice of $\overline T$ does not depend on the terminal condition $u^T$ and therefore we can iterate the construction of the solution on subintervals $[T-k\overline{T},T-(k-1)\overline{T}]$ for $k\in {1,\dots,n}$ and $n\in\N$ such that $T-n\overline{T}\leqslant 0$. Here, we choose the terminal condition of the solution on $[T-k\overline{T},T-(k-1)\overline{T}]$ equal to the initial value of the solution constructed in the previous iteration step. We then obtain the solution $u\in\mathcal{L}_{T}^{0,\alpha+\beta}$ on $[0,T]$ by patching the solutions on the subintervals together. Indeed, $u$ is the fixed point of $\Phi^{T,T}$, due to the semigroup property $P_{t}P_{s}=P_{t+s}$ for $t,s\geqslant 0$.\\ 
The continuity of the solution map follows from 
\begin{align*}
\norm{u}_{\mathcal{L}_{T}^{0,\alpha+\beta}}\leqslant \sum_{k=0}^{n}\norm{u}_{\mathcal{L}_{\overline{T},T-k\overline{T}}^{0,\alpha+\beta}}
\end{align*} for $n\in\N$ such that $T-(n+1)\overline{T}\leqslant 0$, together with \eqref{eq:more-reg} applied for each of the terms on the right-hand-side and the contraction property on each of the spaces $\mathcal{L}_{\overline{T},T-k\overline{T}}^{\gamma,\alpha+\beta}$.
For a terminal condition $u^{T}\in\calC^{(1-\gamma)\alpha+\beta}_{p}$, the above arguments show that we obtain a solution in $\mathcal{L}_{T}^{\gamma,\alpha+\beta}$. Notice that the blow-up just occures for the solution on the last subinterval $[T-\overline{T},T]$. That is, the solutions on $[T-k\overline{T},T-(k-1)\overline{T}]$ for $k=2,\dots,n$ have a regular terminal condition in $\mathcal{C}^{\alpha+\beta}_{p}$.  
\end{proof}
\noindent Next, we define the space of enhanced distributions and afterwards the solution space for solving the generator equation with paracontrolled terminal condition and right hand side in the rough regime $\beta\leqslant\frac{1-\alpha}{2}$. 
For that, we define for a Banach space $X$, the blow up space
\begin{align*}
\mathcal{M}_{\mathring{\Delta}_{T}}^{\gamma}X=\{g:\mathring{\Delta}_{T}\to X\mid\sup_{0\leqslant s<t\leqslant T}(t-s)^{\gamma}\norm{g(s,t)}_{X}<\infty\}
\end{align*} for the triangle without diagonal $\mathring{\Delta}_{T}:=\{(s,t)\in [0,T]^{2}\mid s<t\}$. Below we take $g(s,t)=P_{t-s}(\partial_{j}\eta^{i}_{t})\reso\eta^{j}_{s}$ for $\eta\in C_T C^{\infty}_{b}(\R^{d},\R^{d})$ and $i,j\in\{1,...,d\}$.
\begin{definition}[Enhanced drift]\label{def:roughdist}
Let $T>0$. For $\beta\in (\frac{2-2\alpha}{3},\!\frac{1-\alpha}{2}]$ and $\gamma\in [\frac{2\beta+2\alpha-1}{\alpha},1)$, we define the space of enhanced drifts $\mathcal{X}^{\beta,\gamma}$ as
the closure of
\begin{align*}
\{(\eta,\calK (\eta)):= \paren[\big]{\eta,\paren[\big]{\sum_{i=1}^{d}P_{\cdot}(\partial_{i}\eta^{j})\reso\eta^{i}}_{j=1,...,d}}:\, \eta\in C_T C^{\infty}_{b}(\R^{d},\R^{d})\}
\end{align*}
in $\CTcalC^{\beta+(1-\gamma)\alpha}_{\R^{d}}\times \mathcal{M}_{\mathring{\Delta}_{T}}^{\gamma}\calC^{2\beta+\alpha-1}_{\R^{d\times d}}$. 
We say that $\mathcal{V}$ is a lift or an enhancement of $V$ if $\mathcal{V}_{1}=V$ and we also write $V\in\mathcal{X}^{\beta,\gamma}$ identifying $V$ with $(\calV_{1},\calV_{2})$.\\
For $\beta\in (\frac{1-\alpha}{2},0)$ and $\gamma\in [\frac{\beta-1}{\alpha},1)$, we set $\mathcal{X}^{\beta,\gamma}=C_{T}\calC^{\beta+(1-\gamma)\alpha}$.
\end{definition}
\begin{remark}\label{rem:past-enhanced-drift}
For $\calV\in\calX^{\beta,\gamma}$, we assume on the first component $\calV_{1}\in C_{T}\calC^{\beta+\alpha(1-\gamma)}$. We think of $\gamma\sim 1$, that is $\gamma<1$, but very close to $1$.
The assumptions on $\calV$ in particular imply by the semigroup estimates, that $t\mapsto P_{T-t}V_{T}^{i}\in\mathcal{M}_{T}^{\gamma}\calC^{\alpha+\beta}$.
Furthermore, from $\sum_{i}P(\partial_{i} V^{j})\reso V^{i}\in\mathcal{M}_{\mathring{\Delta}_{T}}^{\gamma}\calC^{2\beta+\alpha-1}$ follows that $t\mapsto \sum_{i}J^{T}(\partial_{i} V^{j})_{t}\reso V^{i}_{t}=\sum_{i}\int_{t}^{T}P_{r-t}(\partial_{i} V_{r}^{j})\reso V_{t}^{i}dr\in C_{T}\calC^{2\beta+\alpha-1}$.
Indeed, as $\gamma<1$, we can estimate 
\begin{align*}
\sup_{t\in[0,T]}\norm[\big]{\sum_{i=1}^{d}J^{T}(\partial_{i} V^{j})_{t}\reso V^{i}_{t}}_{\alpha+2\beta-1}&\leqslant\sup_{t\in[0,T]}\int_{t}^{T}\norm[\big]{\sum_{i}P_{u-t}(\partial_{i} V_{u}^{j})\reso V_{t}^{i}}_{\alpha+2\beta-1}du\\&\leqslant\norm[\big]{\sum_{i}P_{\cdot}(\partial_{i} V^{j})\reso V^{i}}_{\mathcal{M}_{\mathring{\Delta}_{T}}^{\gamma}\calC^{2\beta+\alpha-1}_{\R^{d}}}\sup_{t\in[0,T]}\int_{t}^{T}(u-t)^{-\gamma}du\\&\lesssim\norm[\big]{\sum_{i}P_{\cdot}(\partial_{i} V^{j})\reso V^{i}}_{\mathcal{M}_{\mathring{\Delta}_{T}}^{\gamma}\calC^{2\beta+\alpha-1}_{\R^{d}}}\times T^{1-\gamma},
\end{align*} using that $\gamma<1$. 
Analogously we obtain that $\sum_{i}J^{r}(\partial_{i} V^{j})\reso V^{i}\in C_{[0,r]}\calC^{\alpha+2\beta-1}$ with a uniform bound in $r\in (0,T]$. The assumptions on the enhancement will become handy, as soon as we consider paracontrolled solutions on subintervals of $[0,T]$.
\end{remark}
\begin{remark}
We assume the lower bound on $\gamma$ to ensure, that the regularity of $V$, respectively the regularity of the resonant products $\sum_{i}J^{T}(\partial_{i} V^{j})_{t}\reso V^{i}_{t}$ are negative. 
That is, for $\gamma<(2\beta+2\alpha-1)/\alpha$, we obtain that $\sum_{i}J^{T}(\partial_{i} V^{j})_{t}\reso V^{i}_{t}\in C_{T}\calC^{2\beta+(2-\gamma)\alpha-1}$ due to $V\in C_{T}\calC^{\beta+(1-\gamma)\alpha}$ with $2\beta+(2-\gamma)\alpha-1\geqslant 0$.
In this case, $V$ has enough regularity, so that the Kolmogorov PDE can be solved with the classical approach. We exclude this case here, as we explicitly treat the singular case.  
\end{remark}

\begin{definition}\label{def:paradist}
Let $\alpha\in (1,2]$ and $\beta\in(\frac{2-2\alpha}{3},\frac{1-\alpha}{2}]$. Let $T>0$ and $V\in\cal{X}^{\beta,\gamma'}$ for $\gamma'\in [\frac{2\beta+2\alpha-1}{\alpha},1)$ and let $u^{T,\prime}\in\calC^{\alpha+\beta-1}_p$. 
For $\gamma\in (\gamma',\frac{\alpha}{2-\alpha-3\beta}\gamma')$ and $\overline{T}\in (0,T]$, we define the space of paracontrolled distributions $\mathcal{D}_{\overline{T},T}^{\gamma} = \mathcal{D}_{\overline{T},T}^{\gamma,\gamma'}(\calV, u^{T,\prime})$ as the set of tuples $(u,u')\in \mathcal{L}_{\overline{T},T}^{\gamma',\alpha+\beta}\times (\mathcal{L}_{\overline{T},T}^{\gamma,\alpha+\beta-1})^{d}$, such that 
\begin{align*}
u^{\sharp}:=u-u'\para J^{T}(V)-u^{T,\prime}\para P_{T-\cdot}V_{T}\in \mathcal{L}_{\overline{T},T}^{\gamma,2(\alpha+\beta)-1}.
\end{align*} 
We define a metric on $\mathcal{D}_{\overline{T},T}^{\gamma}$ by 
\begin{align*}
d_{\mathcal{D}_{\overline{T},T}^{\gamma}}((u,u'),(v,v')) &:=\norm{u-v}_{\mathcal{D}_{\overline{T},T}^{\gamma}} \\&:= \norm{u-v}_{\mathcal{L}_{\overline{T},T}^{\gamma',\alpha+\beta}}
		+\norm{u'-v'}_{(\mathcal{L}_{\overline{T},T}^{\gamma,\alpha+\beta-1})^{d}}+\norm{u^{\sharp}-v^{\sharp}}_{\mathcal{L}_{\overline{T},T}^{\gamma,2(\alpha+\beta)-1}}.
\end{align*} 
Then, $(\mathcal{D}_{\overline{T},T}^{\gamma},d_{\mathcal{D}_{\overline{T},T}^{\gamma}})$ is a complete metric space. If moreover $(v,v') \in \mathcal{D}_{\overline{T},T}^{\gamma,\gamma'}(\mathcal{W}, v^{T,\prime})$ for different data $(\mathcal{W},v^{T,\prime})\in\calX^{\beta,\gamma'}\times\calC^{\alpha+\beta-1}_p$, then we use the same definition for $\norm{u-v}_{\mathcal{D}_{\overline{T},T}^{\gamma}}$, despite the fact that $(u,u')$ and $(v,v')$ do not live in the same space.
\end{definition}
\begin{remark}
The intuition behind the paracontrolled ansatz is as follows. Assume for simplicitiy regular data $(u^{T},f)\in\calC^{2(\alpha+\beta)-1}_p\times\mathcal{L}_{T}^{0,\alpha+2\beta-1}$.  Assume also that we found a solution $u\in \mathcal{L}_{T}^{0,\alpha+\beta}$
and that we can make sense of the resonant product $\nabla u\reso V$ in such a way that it has its natural regularity $C_{T}\calC^{2\beta+\alpha-1}_p$, despite the fact that $2\beta+\alpha-1\leqslant 0$. Then we would get that
\begin{align*}
u^{\sharp}:&=u-\nabla u\para J^{T}(V)\\&=\squeeze[1]{P_{T-\cdot}u^{T}-J^{T}(f)+J^{T}(\nabla u\arap V)+J^{T}(\nabla u\reso V)+(J^{T}(\nabla u\para V)-\nabla u\para J^{T}(V))}
\end{align*} is more regular than $u$. Indeed, by the Schauder estimates for the first four terms and by the commutator estimate from \cref{lem:sharp}, we obtain that $u^{\sharp}\in\mathcal{L}_{T}^{0,2(\alpha+\beta)-1}$. This explains why the paracontrolled ansatz might be justified. The reason why the ansatz is useful is that it isolates the singular part of $u$ in a paraproduct, that we can handle by commutator estimates and the assumptions on $V$.
\end{remark}
\noindent Our main theorem of this section is the following. We give its proof after the corollary below.
\begin{theorem}\label{thm:singular}
Let $T>0$, $\alpha\in (1,2]$, $p\in[1,\infty]$ and $\beta\in (\frac{2-2\alpha}{3},\frac{1-\alpha}{2}]$ and $\mathcal{V}\in\calX^{\beta,\gamma'}$ for $\gamma'\in [\frac{2\beta+2\alpha-1}{\alpha},1)$. 
Let 
\begin{align*}
f=f^{\sharp}+f'\para V
\end{align*} for $f^{\sharp}\in \mathcal{L}_{T}^{\gamma',\alpha+2\beta-1}$, $f'\in (\mathcal{L}_{T}^{\gamma',\alpha+\beta-1})^{d}$ and 
\begin{align*}
u^{T}=u^{T,\sharp}+u^{T,\prime}\para V_{T}
\end{align*} for $u^{T,\sharp}\in \calC^{(2-\gamma')\alpha+2\beta-1}_p$, $u^{T,\prime}\in (\calC^{\alpha+\beta-1}_p)^{d}$.\\
Then for $\gamma\in (\gamma',\frac{\alpha}{2-\alpha-3\beta}\gamma')$ there exists a unique mild solution $(u,u^{\prime})\in\mathcal{D}_{T}^{\gamma}(\calV, u^{T,\prime})$ of the singular  Kolmogorov backward PDE
\begin{align*}
\mathcal{G}^{\mathcal{V}}u=f,\qquad u(T,\cdot)=u^{T}.
\end{align*}
\end{theorem}
\begin{remark}
As $\mathcal{L}_{T}^{\tilde{\gamma},\theta}\subset\mathcal{L}_{T}^{\gamma',\theta}$ and $\calC_{p}^{(2-\tilde{\gamma})\alpha+2\beta-1}\subset \calC_{p}^{(2-\gamma')\alpha+2\beta-1}$ for $\tilde{\gamma}\in [0,\gamma']$, we can in particular treat $f^{\sharp}\in \mathcal{L}_{T}^{\tilde{\gamma},\alpha+2\beta-1}$, $f^{\prime}\in \mathcal{L}_{T}^{\tilde{\gamma},\alpha+\beta-1}$ and $u^{T,\sharp}\in \calC_{p}^{(2-\tilde{\gamma})\alpha+2\beta-1}$. 
\end{remark}
\begin{remark}
Examples for right-hand-sides and terminal conditions, which are paracontrolled by $V$, respectively $V_{T}$, are the following.
Clearly we can take as a right-hand side $f=V^{i}$, i.e. $f'=e_{i}$ for the $i$-th unit vector $e_{i}$. Another example would be $f=J^{T}(\nabla V^{i})\cdot V$ for $i\in \{1,\dots,d\}$, where $f^{\sharp}=J^{T}(\nabla V^{i})\reso V+J^{T}(\nabla V^{i})\arap V$ and $f^{\prime}=J^{T}(\nabla V^{i})$. Furthermore, as a terminal condition, we can take $u^{T}=V^{i}_{T}$, i.e. $u^{T,\prime}=e_{i}$. 
\end{remark}
\noindent In the case of $u^{T,\prime}=0$, the terminal condition can still be irregular, but is such that $t\mapsto P_{T-t}u^{T}=P_{T-t}u^{T,\sharp}\in\mathcal{M}_{T}^{\gamma'}\calC^{2(\alpha+\beta)-1}_p$. As  $\frac{2\alpha+2\beta-1}{\alpha}\leqslant\gamma'$ and thus $(2-\gamma')\alpha+2\beta-1\leqslant 0$, another example for a terminal condition, that can be treated with our approach would be a distribution $u^{T}=u^{T,\sharp}\in\calC^{0}_p$. An example would be $u^{T}=\delta_{0}\in\calC^{0}_{1}$, where $\delta_{0}$ denotes the Dirac measure at $x=0$.\\ 
In the case of $u^{T,\prime}=0$ and $u^{T,\sharp}\in\calC^{2(\alpha+\beta)-1}$, the terminal condition is sufficiently regular, such that we can prove, that the solution of the equation is an element of the solution space without blow-up (provided, that $f$ admits zero blow-up).  
We define, in the case of $u^{T,\prime}=0$ and $u^{T,\sharp}\in\calC^{2(\alpha+\beta)-1}$, the paracontrolled solution space as
\begin{align*}
\squeeze[1]{D_{T}:=\mathcal{D}_{T}^{0}=\{(u,u^{\prime})\in\mathcal{L}_{T}^{0,\alpha+\beta}\times(\mathcal{L}_{T}^{0,\alpha+\beta-1})^{d}\mid u^{\sharp}:=u-u^{\prime}\para J^{T}(V)\in\mathcal{L}_{T}^{0,2(\alpha+\beta)-1}\}.} 
\end{align*}
\begin{corollary}[Regular terminal condition]\label{cor:non-singular}
Let $T>0$, $\alpha\in (1,2]$, $p\in[1,\infty]$ and $\beta\in (\frac{2-2\alpha}{3},\frac{1-\alpha}{2}]$ and $\mathcal{V}\in\calX^{\beta,\gamma'}$ for $\gamma'\in [\frac{2\beta+2\alpha-1}{\alpha},1)$. Let $f=f^{\sharp}+f^{\prime}\para V$ for $f^{\sharp}\in\mathcal{L}_{T}^{0,\alpha+2\beta-1}$ and $f^{\prime}\in\mathcal{L}_{T}^{0,\alpha+\beta-1}$ and let $u^{T}=u^{T,\sharp}\in \calC^{2\alpha+2\beta-1}_p$ be non-singular.\\
Then, there exists a unique mild solution $u\in D_{T}$ of the generator equation
\begin{align*}
\mathcal{G}^{\mathcal{V}}u=f,\qquad u(T,\cdot)=u^{T}.
\end{align*}
\end{corollary}
\noindent The proof is deferred to page \pageref{proof:cor:non-singular}.
\begin{remark}\label{rem:other-ed-ass}
The proof of the corollary only uses that $\calV=(V, (J^{T}(\partial_{i}V^{j})\reso V^{i})_{i,j})\in C_{T}\calC^{\beta}_{\R^{d}}\times C_{T}\calC^{2\beta+\alpha-1}_{\R^{d\times d}}$, which is implied by the stronger assumption $\calV\in\calX^{\beta,\gamma'}$ (cf. \cref{rem:past-enhanced-drift}). 
\end{remark}
\begin{proof}[Proof of \cref{thm:singular}]
Let $\calV\in\calX^{\beta,\gamma'}$ with $\calV_{1}=V$, $\calV_{2}=(\sum_{i}P_{\cdot}(\partial_{i}V^{j})\reso V^{i})_{j}$ for $\gamma'\in [\frac{2\beta+2\alpha-1}{\alpha},1)$. Let $\overline{T}\in (0,T]$ to be chosen later and $\gamma\in (\gamma',\frac{\alpha}{2-\alpha-3\beta}\gamma')$. Then we define the contraction mapping as 
\begin{align}
\phi=\phi^{\overline{T},T}:\mathcal{D}^{\gamma}_{\overline{T},T}\to\mathcal{D}^{\gamma}_{\overline{T},T},\quad (u,u^{\prime})\mapsto (\psi(u),\nabla u-f^{\prime})
\end{align} for 
\begin{align}\label{eq:decomp-Keq}
\psi(u)(t)&=P_{T-t}u^{T}+J^{T}(-f)(t)+J^{T}(\nabla u\cdot \calV)(t),\quad t\in [T-\overline{T},T]\nonumber
\\&=P_{T-\cdot}u^{T,\sharp}+J^{T}(-f^{\sharp})+J^{T}(\nabla u\reso \calV)+J^{T}(V\para \nabla u)\nonumber\\&\quad+C_{1}(u^{T,\prime},V_{T})+C_{2}(-f^{\prime},V)+C_{2}(\nabla u,V)\nonumber\\&\quad+(\nabla u-f^{\prime})\para J^{T}(V)+u^{T,\prime}\para P_{T-\cdot}V_{T},
\end{align} where we define
\begin{align}
\nabla u\reso \calV&=\sum_{i=1}^{d}\partial_{i}u\reso\calV^{i}\nonumber\\&:=\sum_{i=1}^{d}[u^{\prime}\cdot(J^{T}(\partial_{i} V)\reso V^{i})+C_{3}(u^{\prime},J^{T}(\partial_{i} V),V^{i})+U^{\sharp}\reso V^{i}\label{eq:resoproduct1}\\&\qquad+u^{T,\prime}\para (P_{T-\cdot} \partial_{i} V_{T}\reso V^{i})+C_{3}(u^{T,\prime},P_{T-\cdot}\partial_{i} V_{T},V^{i})]\label{eq:resoproduct2},
\end{align} with $U^{\sharp}:=\partial_{i} u^{\sharp}+\partial_{i} u^{\prime}\para J^{T}(V)+\partial_{i}u^{T,\prime}\para P_{T-\cdot}V_{T}$. The commutators are defined as follows:
\begin{align*}
C_{1}(f,g):=P_{T-\cdot}(f\para g)-f\para P_{T-\cdot}g, \quad C_{2}(u,v):=J^{T}(u\para v)-u\para J^{T}(v),
\end{align*} where $C_{1}$ denotes the commutator on the semigroup $P_{T-\cdot}$ and $C_{2}$ is the commutator from \cref{lem:sharp}. Furthermore, $C_{3}$ denotes the commuator from \cite[Lemma 2.4]{Gubinelli2015Paracontrolled}, that is 
\begin{align*}
C_{3}(f,g,h):=(f\para g)\reso h-f(g\reso h).
\end{align*} 
For the terms in \eqref{eq:resoproduct1}, we obtain with \cref{rem:past-enhanced-drift}, the paraproduct estimates and \cite[Lemma 2.4]{Gubinelli2015Paracontrolled} using that $3\beta+2\alpha-2>0$ and $2\beta+\alpha-1\leqslant 0$,
\begin{align*}
\MoveEqLeft
\norm[\big]{\sum_{i}\paren[\big]{u^{\prime}\cdot(J^{T}(\partial_{i} V)\reso V^{i})+C_{3}(u^{\prime},J^{T}(\partial_{i} V),V^{i})+U^{\sharp}\reso V^{i}}}_{\mathcal{M}_{\overline{T},T}^{\gamma'}\calC^{\alpha+2\beta-1}}\\&\lesssim \norm{\calV}_{\calX^{\beta,\gamma}}(1+\norm{\calV}_{\calX^{\beta,\gamma}})[\norm{u^{\sharp}}_{\mathcal{M}_{\overline{T},T}^{\gamma'}\calC^{2(\alpha+\beta)-1}_{p}}+\norm{u^{\prime}}_{\mathcal{M}_{\overline{T},T}^{\gamma'}(\calC^{\alpha+\beta-1}_{p})^{d}}]
\end{align*} 
For the terms in \eqref{eq:resoproduct2}, we have by the estimate on the paraproduct and the definition of the enhanced distribution space $\calX^{\beta,\gamma'}$
\begin{align*}
\norm[\big]{\sum_{i}u^{T,\prime}\para (P_{T-\cdot} \partial_{i} V_{T}\reso V^{i})}_{\mathcal{M}_{\overline{T},T}^{\gamma'}\calC_{p}^{2\beta+\alpha-1}}&\lesssim\norm{u^{T,\prime}}_{(\calC^{\alpha+\beta-1}_p)^{d}}\norm[\big]{\sum_{i}P_{T-\cdot} \partial_{i} V_{T}\reso V^{i}}_{\mathcal{M}_{\overline{T},T}^{\gamma'}(\calC^{2\beta+\alpha-1})^d}\\&\lesssim\norm{u^{T,\prime}}_{(\calC^{\alpha+\beta-1}_p)^d}\norm{\calV}_{\calX^{\beta,\gamma'}},
\end{align*} where we used that $\alpha+\beta-1>0$. By the commutator estimate for $C_{3}$ from \cite[Lemma 2.4]{Gubinelli2015Paracontrolled} 
and the estimates for the semigroup to control $P_{T-\cdot}\nabla V_{T}$, we obtain
\begin{align*}
\norm{C_{3}(u^{T,\prime},P_{T-\cdot}\partial_{i} V_{T},V^{i})}_{\mathcal{M}_{\overline{T},T}^{\gamma'}\calC_{p}^{3\beta+2\alpha-2}}\lesssim\norm{u^{T,\prime}}_{(\calC^{\alpha+\beta-1}_p)^d}\norm{\calV}_{\calX^{\beta,\gamma'}}^{2},
\end{align*} using again $2\alpha+3\beta-2>0$ by the assumption on $\beta$.\\
Define $\epsilon:=\alpha-\alpha\frac{\gamma'}{\gamma}$.  Then it follows that $\epsilon\in (0,3\beta+2\alpha-2)$  by the assumption on $\gamma$.
Subtracting $\epsilon$ regularity for $u^{\prime}$ and $u^{\sharp}$, we can estimate the resonant product along the same lines as above, due to $3\beta+2\alpha-2-\epsilon>0$, obtaining
\begin{align}\label{eq:trick}
\MoveEqLeft
\norm{\nabla u\reso \calV}_{\mathcal{M}_{\overline{T},T}^{\gamma'}\calC_{p}^{2\beta+\alpha-1}}\nonumber\\&\lesssim \norm{\calV}_{\calX^{\beta,\gamma'}}(1+\norm{\calV}_{\calX^{\beta,\gamma'}})
\paren[\big]{\norm{u^{\sharp}}_{\mathcal{M}^{\gamma'}_{\overline{T},T}\calC_{p}^{2(\alpha+\beta)-1-\epsilon}}+\norm{u'}_{\mathcal{M}^{\gamma'}_{\overline{T},T}(\calC_{p}^{\alpha+\beta-1-\epsilon})^d}}\nonumber\\&\qquad+\norm{\calV}_{\calX^{\beta,\gamma'}}(1+\norm{\calV}_{\calX^{\beta,\gamma'}})\norm{u^{T,\prime}}_{(\calC^{\alpha+\beta-1}_p)^d}
\nonumber\\&\lesssim \norm{\calV}_{\calX^{\beta,\gamma'}}(1+\norm{\calV}_{\calX^{\beta,\gamma'}})[\norm{(u,u^{\prime})}_{\mathcal{D}_{\overline{T},T}^{\gamma}}+\norm{u^{T,\prime}}_{(\calC^{\alpha+\beta-1}_p)^d}].
\end{align} 
In \eqref{eq:trick}, we moreover used the interpolation bound \eqref{eq:i3} for the norm of $u^{\prime}$, that is
\begin{align*}
\norm{u^{\prime}}_{\mathcal{M}^{\gamma'}_{\overline{T},T}\calC_{p}^{\alpha+\beta-1-\epsilon}}=\norm{u^{\prime}}_{\mathcal{M}^{\gamma(1-\epsilon/\alpha)}_{\overline{T},T}\calC_{p}^{\alpha+\beta-1-\epsilon}}\lesssim\norm{u^{\prime}}_{\mathcal{L}^{\gamma,\alpha+\beta-1}_{\overline{T},T}}
\end{align*}  by the definition of $\epsilon$, and analogously for $u^{\sharp}$.
For $(u,u^{\prime}),(v,v^{\prime})\in\mathcal{D}_{\overline{T},T}^{\gamma}(\calV, u^{T,\prime})$, this also implies the Lipschitz bound:
\begin{align*}
\MoveEqLeft
\norm{\nabla u\reso \calV-\nabla v\reso \calV}_{\mathcal{M}_{\overline{T},T}^{\gamma'}\calC_{p}^{2\beta+\alpha-1}}\\&\lesssim \norm{\calV}_{\calX^{\beta,\gamma'}}(1+\norm{\calV}_{\calX^{\beta,\gamma'}})\norm{(u,u^{\prime})-(v,v^{\prime})}_{\mathcal{D}_{\overline{T},T}^{\gamma}}.
\end{align*}
Next, we show that indeed $\phi(u,u^{\prime})=(\psi(u),\nabla u-f^{\prime})\in\mathcal{D}^{\gamma}_{\overline{T},T}$ and that $\phi$ is a contraction for small enough $\overline{T}$.\\ 
Towards the first aim, we note that by \eqref{eq:decomp-Keq},
\begin{align*}
\phi(u,u^{\prime})^{\sharp}&=\psi(u)-(\nabla u-f^{\prime})\para J^{T}(V)-u^{T,\prime}\para P_{T-\cdot}V_{T}\\&=P_{T-\cdot}u^{T,\sharp}+J^{T}(-f^{\sharp})+J^{T}(\nabla u\reso V)+J^{T}(V\para \nabla u)\\&\qquad+C_{1}(u^{T,\prime},V_{T})+C_{2}(-f^{\prime},V)+C_{2}(\nabla u,V).
\end{align*}
By the Schauder estimates, we obtain $P_{T-\cdot}u^{T,\sharp}+J^{T}(f^{\sharp})\in 
\mathcal{L}_{\overline{T},T}^{\gamma',2\alpha+2\beta-1}$ 
and 
\begin{align*}
\MoveEqLeft
\norm{J^{T}(\nabla u\reso \calV)+J^{T}(V\para \nabla u)}_{\mathcal{L}^{\gamma,2\alpha+2\beta-1}_{\overline{T},T}}\\&\lesssim\overline{T}^{\gamma-\gamma'}[\norm{\nabla u\reso \calV}_{\mathcal{M}^{\gamma'}_{\overline{T},T}\calC_{p}^{\alpha+2\beta-1}}+\norm{V\para\nabla u}_{\mathcal{M}^{\gamma'}_{\overline{T},T}\calC_{p}^{\alpha+2\beta-1}}]
\\&\lesssim\overline{T}^{\gamma-\gamma'}\norm{\calV}_{\calX^{\beta,\gamma'}}(1+\norm{\calV}_{\calX^{\beta,\gamma'}})\norm{(u,u^{\prime})}_{\mathcal{D}_{\overline{T},T}^{\gamma}}\\&\qquad+ \overline{T}^{\gamma-\gamma'}\norm{ u}_{\mathcal{L}^{\gamma',\alpha+\beta}_{\overline{T},T}}\norm{V}_{C_{T}\calC^{\beta}_{\R^{d}}}
\end{align*} using the estimate for the resonant product from above.
Utilizing the commutator estimate (\cref{lem:sharp}), we obtain
\begin{align*}
\norm{C_{2}(\nabla u,V)}_{\mathcal{L}^{\gamma,2(\alpha+\beta)-1}_{\overline{T},T}}\lesssim\overline{T}^{\gamma-\gamma'}\norm{V}_{C_{T}\calC^{\beta}_{\R^{d}}}\norm{u}_{\mathcal{L}_{\overline{T},T}^{\gamma',\alpha+\beta}}\lesssim\overline{T}^{\gamma-\gamma'}\norm{V}_{C_{T}\calC^{\beta}_{\R^{d}}}\norm{(u,u^{\prime})}_{\mathcal{D}_{\overline{T},T}^{\gamma}}.
\end{align*} 
By $V_{T}\in\calC^{\beta+(1-\gamma')\alpha}_{\R^{d}}$ and $u^{T,\prime}\in(\calC_{p}^{\alpha+\beta-1})^d$ and the commutator estimate \eqref{schaudercom} for $C_{1}$ for $\vartheta=\gamma'\alpha$ and $\alpha+\beta-1\in (0,1)$ and again \cref{lem:sharp} for $C_{2}$, we have that
\begin{align*}
\MoveEqLeft
\norm{C_{1}(u^{T,\prime},V_{T})+C_{2}(f^{\prime},V)}_{\mathcal{L}^{\gamma',2(\alpha+\beta)-1}_{\overline{T},T}}\\&\lesssim\norm{V}_{C_{T}\calC^{\beta+(1-\gamma')\alpha}_{\R^{d}}}(\norm{u^{T,\prime}}_{(\calC_{p}^{\alpha+\beta-1})^d}+\norm{f^{\prime}}_{\mathcal{L}_{T}^{\gamma',\alpha+\beta-1}}).
\end{align*}
Hence, together we obtain $\phi(u,u^{\prime})^{\sharp}\in\mathcal{L}^{\gamma,2\alpha+2\beta-1}_{\overline{T},T}$.\\
Next, we show that $\psi(u)\in\mathcal{L}^{\gamma',\alpha+\beta}_{\overline{T},T}$.\\ 
Define $\gamma'':=\gamma'(1-\epsilon_{1}/\alpha)$ for a fixed $\epsilon_{1}\in (0,(\alpha+\beta-1)\wedge \frac{(1-\gamma')\alpha}{2-\alpha-3\beta})=(0,\frac{(1-\gamma')\alpha}{2-\alpha-3\beta})$ and define $\epsilon_{2}:=\alpha-\alpha\frac{\gamma''}{\gamma}$. Then it follows that $\epsilon_{2}\in (0,3\beta+2\alpha-2+(1-\gamma')\alpha)$.
Using that $V\in C_{T}\calC^{\beta+(1-\gamma')\alpha}_{\R^{d}}$ and applying twice the interpolation bound \eqref{eq:i3} (once for $u$ and once for $u^{\sharp}$ and $u^{\prime}$), an analogue estimate as for the resonant product $\norm{J^{T}(\nabla u\reso\calV)}_{\mathcal{L}_{T}^{\gamma,2\alpha+2\beta-1}}$ yields that
\begin{align*}
\MoveEqLeft
\norm{J^{T}(\nabla u\cdot \calV)}_{\mathcal{L}_{\overline{T},T}^{\gamma',\beta+\alpha}}\\&\lesssim\norm{J^{T}(\nabla u\para V)}_{\mathcal{L}_{\overline{T},T}^{\gamma',\beta+\alpha}}+\norm{J^{T}(\nabla u\arap V+\nabla u \reso\calV)}_{\mathcal{L}_{\overline{T},T}^{\gamma',2\alpha+\beta-1}}\\&\lesssim \overline{T}^{\gamma'-\gamma''}[\norm{\nabla u\para V}_{\mathcal{M}_{\overline{T},T}^{\gamma''}\calC_{p}^{\beta}}+\norm{\nabla u\arap V+ \nabla u\reso\calV}_{\mathcal{M}_{\overline{T},T}^{\gamma''}\calC_{p}^{\alpha+2\beta-1}}]\\&\lesssim \overline{T}^{\gamma'-\gamma''}\norm{\calV}_{\calX^{\beta,\gamma'}}(1+\norm{\calV}_{\calX^{\beta,\gamma'}})\\&\qquad\times[\norm{u}_{\mathcal{M}_{\overline{T},T}^{\gamma''}\calC_{p}^{\alpha+\beta-\epsilon_{1}}}+\norm{u^{\sharp}}_{\mathcal{M}_{\overline{T},T}^{\gamma''}\calC_{p}^{2(\alpha+\beta)-1-\epsilon_{2}}}+\norm{u^{\prime}}_{\mathcal{M}_{\overline{T},T}^{\gamma''}(\calC_{p}^{\alpha+\beta-1-\epsilon_{2}})^d}]\\&\lesssim \overline{T}^{\gamma'-\gamma''}\norm{\calV}_{\calX^{\beta,\gamma'}}(1+\norm{\calV}_{\calX^{\beta,\gamma'}})[\norm{u}_{\mathcal{L}_{\overline{T},T}^{\gamma',\alpha+\beta}}+\norm{u^{\sharp}}_{\mathcal{L}_{\overline{T},T}^{\gamma,2(\alpha+\beta)-1}}+\norm{u^{\prime}}_{(\mathcal{L}_{\overline{T},T}^{\gamma,\alpha+\beta-1})^d}]\\&=\overline{T}^{\gamma'-\gamma''}\norm{\calV}_{\calX^{\beta,\gamma'}}(1+\norm{\calV}_{\calX^{\beta,\gamma'}})\norm{u}_{\mathcal{D}_{\overline{T},T}^{\gamma,\alpha+\beta}}.
\end{align*} 
Thus, we obtain that
\begin{align*}
\MoveEqLeft
\norm{\psi(u)}_{\mathcal{L}^{\gamma',\alpha+\beta}_{\overline{T},T}}\\&=\norm{P_{T-\cdot}u^{T}+J^{T}(f)+J^{T}(\nabla u\cdot \calV)}_{\mathcal{L}^{\gamma',\alpha+\beta}_{\overline{T},T}}\\&\leqslant\norm{P_{T-\cdot}u^{T}}_{\mathcal{L}^{\gamma',\alpha+\beta}_{\overline{T},T}}+\norm{f}_{\mathcal{L}_{T}^{\gamma',\beta}}+\norm{J^{T}(\nabla u\cdot \calV)}_{\mathcal{L}^{\gamma',\alpha+\beta}_{\overline{T},T}}\\&\lesssim\squeeze[1]{\norm{u^{T,\sharp}}_{\calC^{(2-\gamma')\alpha+2\beta-1}_p}+\norm{u^{T,\prime}}_{(\calC^{\alpha+\beta-1}_p)^d}\norm{\calV}_{\calX^{\beta,\gamma'}}+\norm{C_{1}(u^{T,\prime},V_{T})}_{\mathcal{M}_{T}^{\gamma'}\calC_{p}^{2\alpha+2\beta-1}}}\\&\quad+\norm{f}_{\mathcal{L}_{T}^{\gamma',\beta}}+\overline{T}^{\gamma'-\gamma''}\norm{\calV}_{\calX^{\beta,\gamma'}}(1+\norm{\calV}_{\calX^{\beta,\gamma'}})\norm{u}_{\mathcal{D}_{\overline{T},T}^{\gamma,\alpha+\beta}},
\end{align*} which yields in particular $\psi(u)\in\mathcal{L}^{\gamma',\alpha+\beta}_{\overline{T},T}$. 
The Gubinelli derivative $\phi(u,u^{\prime})^{\prime}=\nabla u-f^{\prime}$, we estimate as follows
\begin{align*}
\MoveEqLeft
\norm{\nabla u-f^{\prime}}_{(\mathcal{L}^{\gamma,\alpha +\beta-1}_{\overline{T},T})^d}\\&\lesssim\norm{\nabla u}_{\mathcal{M}^{\gamma}_{\overline{T},T}(\calC_{p}^{\alpha+\beta-1})^d}+\norm{\nabla u}_{C^{1-\gamma}_{\overline{T},T}(\calC_{p}^{\beta-1})^d}+\norm{\nabla u}_{C^{\gamma,1}_{\overline{T},T}(\calC_{p}^{\beta-1})^d}+\norm{f^{\prime}}_{(\mathcal{L}^{\gamma,\alpha +\beta-1}_{\overline{T},T})^d}
\\&\lesssim \overline{T}^{\gamma-\gamma'}\paren[\big]{\norm{u}_{\mathcal{L}^{\gamma',\alpha+\beta}_{\overline{T},T}}+\norm{f^{\prime}}_{(\mathcal{L}^{\gamma',\alpha +\beta-1}_{\overline{T},T})^d}}\\&\lesssim \overline{T}^{\gamma-\gamma'}\paren[\big]{\norm{u}_{\mathcal{D}^{\gamma}_{\overline{T},T}}+\norm{f^{\prime}}_{(\mathcal{L}^{\gamma',\alpha +\beta-1}_{\overline{T},T})^d}},
\end{align*} where we exploit the fact that $\gamma-\gamma'>0$ to obtain a non-trivial factor depending on $\overline{T}$. Together with the estimate for $\psi(u)$ and $\phi(u,u^{\prime})^{\sharp}$, this yields $\phi(u,u^{\prime})=(\psi(u),\nabla u-f^{\prime})\in\mathcal{D}_{\overline{T},T}^{\gamma}$.\\
The contraction property follows using the above estimates for $\psi(u)$, $\phi(u,u^{\prime})^{\sharp}$ and $\phi(u,u^{\prime})^{\prime}$, utilizing linearity of $\phi$ and $\psi$ (for $u^{T}=0, f=0$), such that
\begin{align}\label{eq:contractionest}
\MoveEqLeft
\norm{(\psi(u),\nabla u-f^{\prime})-(\psi(v),\nabla v-f^{\prime})}_{\mathcal{D}^{\gamma}_{\overline{T},T}}\nonumber\\&=\squeeze[1]{\norm{\psi(u)-\psi(v)}_{\mathcal{L}^{\gamma',\alpha+\beta}_{\overline{T},T}}+\norm{\nabla u-\nabla v}_{(\mathcal{L}^{\gamma,\alpha+\beta-1}_{\overline{T},T})^d}+\norm{\phi(u,u^{\prime})^{\sharp}-\phi(v,v^{\prime})^{\sharp}}_{\mathcal{L}^{\gamma,2\alpha+2\beta-1}_{\overline{T},T}}}\nonumber\\&\lesssim(\overline{T}^{\gamma-\gamma'}\vee\overline{T}^{\gamma'-\gamma''})\norm{\calV}_{\calX^{\beta,\gamma'}}(1+\norm{\calV}_{\calX^{\beta,\gamma'}})\norm{(u,u^{\prime})-(v,v^{\prime})}_{\mathcal{D}_{\overline{T},T}^{\gamma}}.
\end{align} 
Now, we can choose $\overline{T}$ small enough, such that the implicit constant times the factor $(\overline{T}^{\gamma-\gamma'}\vee\overline{T}^{\gamma'-\gamma''})\norm{\calV}_{\calX^{\beta,\gamma'}}(1+\norm{\calV}_{\calX^{\beta,\gamma'}})$ is strictly less than $1$, such that $\phi=\phi^{\overline{T},T}$ is a contraction on the corresponding space $\mathcal{D}_{\overline{T},T}^{\gamma}$. It is left to show, that we can obtain a paracontrolled solution in $\mathcal{D}_{T}^{\gamma}$ on the whole interval $[0,T]$. The solution on $[0,T]$ is obtained by patching the solutions on the subintervals of length $\overline{T}$ together. Indeed, let inductively $u^{[T-\overline{T},T]}$ be the solution on the subinterval $[T-\overline{T},T]$ with terminal condition $u^{T}$ and $u^{[T-k\overline{T},T-(k-1)\overline{T}]}$ be the solution on $[T-k\overline{T},T-(k-1)\overline{T}]$ with terminal condition $u^{[T-k\overline{T},T-(k-1)\overline{T}]}_{T-(k-1)\overline{T}}=u^{[T-(k-1)\overline{T},T-(k-2)\overline{T}]}_{T-(k-1)\overline{T}}$ for $k=2,\dots, n$ and $n\in\N$, such that $T-n\overline{T}\leqslant 0$. There is a small subtlety, as we consider the solution on $[T-k\overline{T},T-(k-1)\overline{T}]$, that is paracontrolled by $J^{T}(V)$ (and not by $J^{T-(k-1)\overline{T}}(V)$). That is, for $k=2,\dots,n$, the solution has the paracontrolled structure,
\begin{align*}
\hspace{1em}&\hspace{-1em}
u^{[T-k\overline{T},T-(k-1)\overline{T}],\sharp}_{t}\\&=\squeeze[1]{u^{[T-k\overline{T},T-(k-1)\overline{T}]}_t-(\nabla u^{[T-k\overline{T},T-(k-1)\overline{T}]}_t-f^{\prime}_t)\para J^{T}(V)_{t}- u^{T,\prime}\para P_{T-t}V_{T}\in\mathcal{L}_{\overline{T},T-(k-1)\overline{T}}^{\gamma,2(\alpha+\beta)-1}}
\end{align*} 
Notice, that for $k\geqslant 2$, $u^{T,\prime}\para P_{T-t}V_{T}\in\mathcal{L}_{\overline{T},T-(k-1)\overline{T}}^{\gamma,2(\alpha+\beta)-1}$, so that term can also be seen as a part of the regular paracontrolled remainder.\\
By assumption we have that $f^{\sharp}\in\mathcal{L}_{T}^{\gamma',\alpha+2\beta-1}$ and $f^{\prime}\in(\mathcal{L}_{T}^{\gamma',\alpha+\beta-1})^d$. This implies by \eqref{eq:subinterval-blow-up} that 
\begin{align*}
&f^{\sharp}\in \mathcal{M}^{0}_{\overline{T},T-(k-1)\overline{T}}\calC^{\alpha+2\beta-1}_p\cap C^{1}_{\overline{T},T-(k-1)\overline{T}}\calC^{2\beta-1}_p,\\& f^{\prime}\in \mathcal{M}^{0}_{\overline{T},T-(k-1)\overline{T}}(\calC^{\alpha+\beta-1}_p)^d\cap C^{1}_{\overline{T},T-(k-1)\overline{T}}(\calC^{\beta-1}_p)^d
\end{align*} for $k=2,\dots,n$. 
If $u^{[T-\overline{T},T]}$ denotes the solution on $[T-\overline{T},T]$, then $u^{[T-\overline{T},T]}_{T-\overline{T}}\in\calC^{\alpha+\beta}_p$ and 
$u^{[T-\overline{T},T],\sharp}_{T-\overline{T}}\in\calC^{2\alpha+2\beta-1}_p$. Thus, for the solution on $[T-2\overline{T},T-\overline{T}]$ follows
\begin{align*}
u^{[T-2\overline{T},T-\overline{T}],\sharp}_{T-\overline{T}}&=u^{[T-2\overline{T},T-\overline{T}]}_{T-\overline{T}}-(\nabla u^{[T-2\overline{T},T-\overline{T}]}_{T-\overline{T}}-f^{\prime}_{T-\overline{T}})\para J^{T}(V)_{T-\overline{T}}-u^{T,\prime}\para P_{\overline{T}}V
\\&=u^{[T-\overline{T},T],\sharp}_{T-\overline{T}}\in\calC^{2\alpha+2\beta-1}_p.
\end{align*} 
Because we can trivially bound,
\begin{align*}
\sup_{t\in[T-2\overline{T},T-\overline{T}]}\norm{P_{T-\overline{T}-t}u^{[T-2\overline{T},T-\overline{T}],\sharp}_{T-\overline{T}}}_{\calC^{2(\alpha+\beta)-1}_p}\lesssim \norm{u^{[T-2\overline{T},T-\overline{T}],\sharp}_{T-\overline{T}}}_{\calC^{2(\alpha+\beta)-1}_p},
\end{align*} 
there is no blow-up for the solution on $[T-2\overline{T},T-\overline{T}]$ at time $t=T-\overline{T}$. Hence, the Banach fixed point argument for the map $\phi^{\overline{T},T-\overline{T}}$ yields a solution $u^{[T-2\overline{T},T-\overline{T}]}\in\mathcal{D}^{\hat{\gamma}}_{\overline{T},T-\overline{T}}$ for any small $\hat{\gamma}>0$. By plugging the solution back in the fixed point map and using the interpolation estimates (cf. the arguments in the proof of \cref{thm:ygeneq} above and \cref{cor:non-singular} below), we obtain that indeed $u^{[T-2\overline{T},T-\overline{T}]}\in\mathcal{D}^{0}_{\overline{T},T-\overline{T}}$. Proceeding iteratively, we thus obtain solutions 
\begin{align*}
u^{[T-k\overline{T},T-(k-1)\overline{T}]}\in\mathcal{D}^{0}_{\overline{T},T-(k-1)\overline{T}}\quad\text{ for }\quad k=2,\dots,n
\end{align*} and $u^{[T-\overline{T},T]}\in\mathcal{D}^{\gamma}_{\overline{T},T-(k-1)\overline{T}}$. 
Then, the solution $u$, which is patched together on the subintervals ($u_{t}:=u^{[T-k\overline{T},T-(k-1)\overline{T}]}_{t}$ for $t\in [T-k\overline{T},T-(k-1)\overline{T}]$, $k=1,\dots,n$), is indeed a fixed point of the map $\phi=\phi^{0,T}$ considered on $[0,T]$ and an element of $\mathcal{D}_{T}^{\gamma}$.
\end{proof}

\begin{proof}[Proof of \cref{cor:non-singular}]\label{proof:cor:non-singular}
By assumption, we have that $u^{T,\prime}=0$ and $u^{T,\sharp}=u^{T}\in\calC^{2(\alpha+\beta)-1}_{p}$ and $f^{\sharp},f^{\prime}$ have no blow-up. By the assumption on $\mathcal{V}$, it follows that $J^{T}(\partial_{i} V^{j})\reso V^{i} \in C_{T}\calC^{\alpha+2\beta-1}$ due to $\gamma'\in (0,1)$. Furthermore due to $u^{T,\prime}=0$ the paraproduct $u^{T,\prime}\para P_{T-\cdot}V_{T}$ in \eqref{eq:decomp-Keq} vanishes, which previously was the term that introduced a blow-up of at least $\gamma'$ for the solution. Thus, we have that $P_{T-\cdot}u^{T}\in C_{T}\calC^{2(\alpha+\beta)-1}$. 
Hence, the arguments from \cref{thm:singular} yield a paracontrolled solution $u\in\mathcal{D}_{T}^{\gamma}$ for any small $\gamma>0$, i.p. $u\in\mathcal{L}_{T}^{\gamma, \alpha+\beta}$. It remains to justify that $u\in D_{T}$. By the regular terminal condition $u^{T}\in\calC^{2(\alpha+\beta)-1}_p\subset\calC^{\alpha+\beta}_p$ and the interpolation estimate \eqref{eq:i3b}, we obtain that 
\begin{align*}
\sup_{t\in[0,T]}\norm{u_{t}}_{\calC^{\alpha+\beta-\alpha\gamma}_p}&\lesssim\norm{u}_{\mathcal{L}_{T}^{\gamma,\alpha+\beta}}+\norm{u_{T}}_{\calC^{\alpha+\beta-\alpha\gamma}_p}\\&\lesssim\norm{u}_{\mathcal{L}_{T}^{\gamma,\alpha+\beta}}+\norm{u_{T}}_{\calC^{\alpha+\beta}_p}
\end{align*} and since $u^{\sharp}_{T}=u_{T}\in\calC^{2(\alpha+\beta)-1}_p$,
\begin{align*}
\sup_{t\in[0,T]}\norm{u^{\sharp}_{t}}_{\calC^{2(\alpha+\beta)-1-\alpha\gamma}_p}&\lesssim\norm{u^{\sharp}}_{\mathcal{L}_{T}^{\gamma,2(\alpha+\beta)-1}}+\norm{u_{T}}_{\calC^{2(\alpha+\beta)-1-\gamma\alpha}_p}\\&\lesssim\norm{u^{\sharp}}_{\mathcal{L}_{T}^{\gamma,2(\alpha+\beta)-1}}+\norm{u_{T}}_{\calC^{2(\alpha+\beta)-1}_p}
\end{align*} for any small $\gamma>0$. If $\gamma$ is small enough, that is $\gamma\in (0,(3\beta+2\alpha-2)/\alpha)$, we can estimate
\begin{align}\label{eq:betterest}
\MoveEqLeft
\sup_{t\in[0,T]}\norm{\nabla u\cdot\calV(t)}_{\beta}\nonumber\\&\lesssim\norm{\calV}_{\calX^{\beta,\gamma'}}(1+\norm{\calV}_{\calX^{\beta,\gamma'}})\paren[\Big]{\sup_{t\in[0,T]}\norm{u_{t}}_{\calC^{\alpha+\beta-\alpha\gamma}_p}+\sup_{t\in[0,T]}\norm{u^{\sharp}_{t}}_{\calC^{2(\alpha+\beta)-1-\alpha\gamma}_p}}.
\end{align} 
Plugging now the solution $u$ back in the contraction map using the fixed point, i.e. $u=P_{T-\cdot}u^{T}+J^{T}(\nabla u\cdot\calV)$, and \eqref{eq:betterest}, we can use the Schauder estimates for $\gamma=\gamma'=0$, such that we obtain that indeed $u\in\mathcal{L}_{T}^{0,\alpha+\beta}$. By the commutator estimate \eqref{eq:c3} for $\gamma=\gamma'=0$ and $u\in\mathcal{L}_{T}^{0,\alpha+\beta}$, we then also obtain that $u^{\sharp}\in\mathcal{L}_{T}^{0,2(\alpha+\beta)-1}$.
\end{proof}
\noindent The next theorem proves the continuity of the solution map. The proof is similar to \cite[Theorem 3.8]{kp}, but adapted to the generalized setting for singular paracontrolled data. There are a few subtleties. First, the space $\mathcal{D}_{T}^{\gamma}(V,u^{T,\prime})$ depends on $V, u^{T,\prime}$. Furthermore due to the blow-up $\gamma>0$, one cannot simply estimate the norm $\mathcal{M}_{T}^{\gamma}\calC^{\theta}$ on the inteval $[0,T]$ by the sum of the respective blow-up norms on subintervals of $[0,T]$. In the case of regular terminal condition, that splitting issue does not occure, but we aim for continuity of the solution map in $\mathcal{L}_{T}^{0,\alpha+\beta}$. This we establish by first proving continuity of the map with values in $\mathcal{L}_{T}^{\gamma,\alpha+\beta}$ for any small $\gamma>0$ and conclude from there together with the interpolation estimates.
\begin{theorem}\label{thm:cont}
In the setting of Theorem~\ref{thm:singular}, the solution map 
\[
	(u^{T}=u^{T,\sharp}+u^{T,\prime}\para V_{T}, \,f=f^{\sharp}+f^{\prime}\para V,\,\calV)  \mapsto (u,u^{\sharp})\in \mathcal{L}_{T}^{\gamma',\alpha+\beta}\times \mathcal{L}_{T}^{\gamma,2(\alpha+\beta)-1} ,
\]
is locally Lipschitz continuous, that is,
\begin{align}\label{eq:Lip-est}
\MoveEqLeft
\norm{u-v}_{\mathcal{L}_{T}^{\gamma',\alpha+\beta}}+\norm{u^{\sharp}-v^{\sharp}}_{\mathcal{L}_{T}^{\gamma,2(\alpha+\beta)-1}}\nonumber\\&\leqslant C[\norm{u^{T,\sharp}-v^{T,\sharp}}_{\calC^{(2-\gamma')\alpha+2\beta-1}_p}+\norm{u^{T,\prime}-v^{T,\prime}}_{(\calC^{\alpha+\beta-1}_p)^{d}}\nonumber\\&\qquad+\norm{f^{\sharp}-g^{\sharp}}_{\mathcal{L}^{\gamma',\alpha+2\beta-1}_{T}}+\norm{f^{\prime}-g^{\prime}}_{(\mathcal{L}^{\gamma',\alpha+\beta-1}_{T})^{d}}+\norm{\calV-\mathcal{W}}_{\calX^{\beta,\gamma'}}]
\end{align} for a constant $C=C(T,\norm{\mathcal V},\norm{\mathcal W},\norm{u^{T}},\norm{v^{T}},\norm{f},\norm{g})>0$.\\
Furthermore, in the setting of \cref{cor:non-singular}, the solution map
\[
	(u^{T}=u^{T,\sharp},\, f=f^{\sharp}+f^{\prime}\para V,\,\calV)  \mapsto (u,u^{\sharp})\in \mathcal{L}_{T}^{0,\alpha+\beta}\times \mathcal{L}_{T}^{0,2(\alpha+\beta)-1} ,
\]
is locally Lipschitz continuous allowing for an analogue bound \eqref{eq:Lip-est} with $\gamma'=0$ for the norms of $u^{T,\sharp}-v^{T,\sharp},f^{\sharp}-g^{\sharp},f^{\prime}-g^{\prime}$ and $u^{T,\prime}=v^{T,\prime}=0$.
\end{theorem}

\begin{proof}
We first prove the continuity in the case of singular paracontrolled data.\\   
Let $u$ be the solution of the PDE for $\mathcal{V}\in\mathcal{X}^{\beta,\gamma'}$, $f=f^{\sharp}+f^{\prime}\para V$ and $u^{T}=u^{T,\sharp}+u^{T,\prime}\para V_{T}$ and $v$ the solution corresponding to the data $\mathcal{W}$, $g$ and $v^{T}$. By the fixed point property we have $\phi(u,u')=(u,u')$ and $\phi(v,v')=(v,v')$ and thus $u'=\nabla u-f'$ and $v'=\nabla v-g'$. Hence, we can estimate 
\begin{align}\label{eq:u-prime-est}
\norm{u'-v'}_{(\mathcal{L}_{T}^{\gamma',\alpha+\beta-1})^d}\lesssim\norm{u-v}_{\mathcal{L}_{T}^{\gamma',\alpha+\beta}}+\norm{f'-g'}_{(\mathcal{L}_{T}^{\gamma',\alpha+\beta-1})^d}.
\end{align} 
We estimate the terms in \eqref{eq:Lip-est} by itself times a factor less than $1$, plus a term depending on $\norm{f-g}$, $\norm{\mathcal{V}-\mathcal{W}}$ and $\norm{u^{T}-v^{T}}$. Here we keep in mind that $u\in\mathcal{D}^{\gamma}_{T}(V,u^{T,\prime})$, whereas $v\in\mathcal{D}^{\gamma}_{T}(W,v^{T,\prime})$, but we explained the notation of $\norm{u - v}_{\mathcal{D}^{\gamma}_{T}}$ in \cref{def:paradist}. For that purpose, we estimate the product using re-bracketing like $ab-cd=a(b-d)+(a-c)d$ and the estimate \eqref{eq:trick} for the product, where $\gamma''<\gamma'$,
\begin{align}\label{eq:product-bound}
\MoveEqLeft
\norm{\nabla u\cdot\mathcal{V}-\nabla v\cdot\mathcal{W}}_{\mathcal{M}_{T}^{\gamma''}\mathcal{C}^{\beta}_p}
\nonumber\\&\lesssim (1+\norm{\mathcal{W}}_{\mathcal{X}^{\beta,\gamma'}})\norm{\mathcal{V}}_{\mathcal{X}^{\beta,\gamma'}}\norm{u-v}_{\mathcal{D}^{\gamma}_{T}}+(1+\norm{\mathcal{W}}_{\mathcal{X}^{\beta,\gamma'}})\norm{\mathcal{V}-\mathcal{W}}_{\mathcal{X}^{\beta,\gamma'}}\norm{v}_{\mathcal{D}^{\gamma}_{T}}\nonumber\\&\qquad + \norm{\mathcal{V}}_{\calX^{\beta,\gamma'}}\norm{u}_{\mathcal{D}^{\gamma}_{T}}\norm{\mathcal{V}-\mathcal{W}}_{\calX^{\beta,\gamma'}}\nonumber\\&\qquad + \tilde{C}(\norm{\calV},\norm{\mathcal{W}},\norm{u^{T,\prime}},\norm{v^{T,\prime}})\paren[\big]{\norm{\mathcal{V}-\mathcal{W}}_{\mathcal{X}^{\beta,\gamma'}}+\norm{u^{T,\prime}-v^{T,\prime}}_{(\calC^{\alpha+\beta-1}_p)^{d}}}.
\end{align}
Since the solution $u$ can be bounded in terms of $u^{T},f,\mathcal{V}$ by Gronwall's inequality for locally finite measures using that $\gamma,\gamma'\in (0,1)$ (cf. \cite[Appendix, Theorem 5.1]{Ethier1986}), and similarly for $v$, we conclude that
\begin{align*}
\hspace{0.6em}&\hspace{-0.6em}
\norm{\nabla u\cdot\mathcal{V}-\nabla v\cdot\mathcal{W}}_{\mathcal{M}_{T}^{\gamma''}\mathcal{C}^{\beta}_p}
\\&\lesssim\paren[\big]{(\norm{\mathcal V}_{\calX^{\beta,\gamma'}}+\norm{v}_{\mathcal{D}^{\gamma}_{T}})(1+\norm{\mathcal W}_{\calX^{\beta,\gamma'}})+\norm{\mathcal V}_{\calX^{\beta,\gamma'}}\norm{u}_{\mathcal{D}^{\gamma}_{T}}}\times\\&\hspace{23em}\paren[\Big]{\norm{u-v}_{\mathcal{D}^{\gamma}_{T}}+\norm{\mathcal{V}-\mathcal{W}}_{\calX^{\beta,\gamma'}}}\\&\qquad +\tilde{C}\paren[\big]{\norm{\mathcal{V}-\mathcal{W}}_{\mathcal{X}^{\beta,\gamma'}}+\norm{u^{T,\prime}-v^{T,\prime}}_{(\calC^{\alpha+\beta-1}_p)^{d}}}\\&\lesssim C\paren[\Big]{\norm{u-v}_{\mathcal{D}^{\gamma}_{T}}+\norm{\mathcal{V}-\mathcal{W}}_{\calX^{\beta,\gamma'}}+\norm{u^{T,\prime}-v^{T,\prime}}_{(\calC^{\alpha+\beta-1}_p)^{d}}},
\end{align*} where $C=C(\norm{\mathcal V},\norm{\mathcal W},\norm{u^{T}},\norm{v^{T}},\norm{f},\norm{g})$ is a constant, that depends on the norms of the input data on $[0,T]$.
Therefore, we obtain by the fixed point and using the estimate for $\norm{\phi(u,u^{\prime})}_{\mathcal{L}_{T}^{\gamma',\alpha+\beta}}$ from the proof of \cref{thm:singular} with $\gamma''<\gamma'$,
\begin{align*}
\MoveEqLeft
\norm{u-v}_{\mathcal{L}_{T}^{\gamma',\alpha+\beta}}\\&\lesssim \norm{u^{T,\sharp}-v^{T,\sharp}}_{\calC^{(2-\gamma')\alpha+2\beta-1}_p}+\norm{u^{T,\prime}-v^{T,\prime}}_{(\calC^{\alpha+\beta-1}_p)^d}\norm{\calV}_{\calX^{\beta,\gamma'}}\nonumber\\&\qquad+\norm{u^{T,\prime}}_{(\calC^{\alpha+\beta-1}_p)^d}\norm{\calV-\mathcal{W}}_{\calX^{\beta,\gamma'}}+\norm{f^{\prime}}_{(\mathcal{L}^{\gamma',\alpha+\beta-1}_T)^d}\norm{\calV-\mathcal{W}}_{\calX^{\beta,\gamma'}}\nonumber\\&\qquad+\norm{f^{\sharp}-g^{\sharp}}_{\mathcal{L}^{\gamma',\alpha+2\beta-1}_T}+\norm{f^{\prime}-g^{\prime}}_{(\mathcal{L}^{\gamma',\alpha+\beta-1}_T)^d}\norm{\calV}_{\calX^{\beta,\gamma'}}\nonumber\\&\qquad+T^{\gamma'-\gamma''}\norm{\nabla u\cdot\mathcal{V}-\nabla v\cdot\mathcal{W}}_{\mathcal{M}_{T}^{\gamma''}\mathcal{C}^{\beta}_p}.
\end{align*}
Moreover, using the fixed point and the estimate for $\norm{\phi(u,u')^{\sharp}}_{\mathcal{L}_{T}^{\gamma,2(\alpha+\beta)-1}}$, we obtain
\begin{align}\label{eq:rgeneq-pr1}
\MoveEqLeft
\norm{u^{\sharp}-v^{\sharp}}_{\mathcal{L}_{T}^{\gamma,2(\alpha+\beta)-1}}\nonumber\\ \nonumber
&\lesssim \norm{u^{T,\sharp}-v^{T,\sharp}}_{\calC^{(2-\gamma')\alpha+2\beta-1}_p}+\norm{u^{T,\prime}-v^{T,\prime}}_{(\calC^{\alpha+\beta-1}_p)^d}\norm{\calV}_{\calX^{\beta,\gamma'}}\\&\qquad+\norm{u^{T,\prime}}_{(\calC^{\alpha+\beta-1}_p)^d}\norm{\calV-\mathcal{W}}_{\calX^{\beta,\gamma'}}+\norm{f^{\prime}}_{(\mathcal{L}^{\gamma',\alpha+\beta-1}_T)^d}\norm{\calV-\mathcal{W}}_{\calX^{\beta,\gamma'}}\nonumber\\&\qquad+\norm{f^{\sharp}-g^{\sharp}}_{\mathcal{L}^{\gamma',\alpha+2\beta-1}_T}+\norm{f^{\prime}-g^{\prime}}_{(\mathcal{L}^{\gamma',\alpha+\beta-1}_T)^d}\norm{\calV}_{\calX^{\beta,\gamma'}}\nonumber\\&\qquad+T^{\gamma-\gamma'}\norm{\nabla u\cdot\mathcal{V}-\nabla v\cdot\mathcal{W}}_{\mathcal{M}_{T}^{\gamma'}\mathcal{C}^{\beta}_p}\nonumber
\\&\qquad+T^{\gamma-\gamma'}\norm{\mathcal{V}-\mathcal{W}}_{\mathcal{X}^{\beta,\gamma'}}\norm{u}_{\mathcal{D}_{T}^{\gamma}} + T^{\gamma'-\gamma}\norm{\mathcal{V}}_{\mathcal{X}^{\beta,\gamma'}}\norm{u-v}_{\mathcal{D}_{T}^{\gamma}}.
\end{align}
To shorten notation, let us abbreviate the term in \eqref{eq:Lip-est}, that we aim to estimate, in  the following by
\begin{align*}
\norm{u-v}_{\gamma,\alpha+\beta}:=\norm{u-v}_{\mathcal{L}_{T}^{\gamma',\alpha+\beta}}+\norm{u^{\sharp}-v^{\sharp}}_{\mathcal{L}_{T}^{\gamma,2(\alpha+\beta)-1}}.
\end{align*}
Then overall, using also \eqref{eq:u-prime-est}, we obtain 
\begin{align*}
\norm{u-v}_{\gamma,\alpha+\beta}&\leqslant C[\norm{u^{T,\sharp}-v^{T,\sharp}}_{\calC^{(2-\gamma')\alpha+2\beta-1}_p}+\norm{u^{T,\prime}-v^{T,\prime}}_{(\calC^{\alpha+\beta-1}_p)^d}\nonumber\\&\qquad+\norm{f^{\sharp}-g^{\sharp}}_{\mathcal{L}^{\gamma',\alpha+2\beta-1}_T}+\norm{f^{\prime}-g^{\prime}}_{(\mathcal{L}^{\gamma',\alpha+\beta-1}_T)^d}+\norm{\calV-\mathcal{W}}_{\calX^{\beta,\gamma'}}]\nonumber \\&\qquad+(T^{\gamma-\gamma'}\vee T^{\gamma'-\gamma''})C\norm{u-v}_{\gamma,\alpha+\beta},
\end{align*} where $C> 0$ is again a (possibly different) constant depending on the norms of the input data.
Assume for the moment that $T$ is small enough so that $(T^{\gamma-\gamma'}\vee T^{\gamma'-\gamma''})C$ times the implicit constant on the right-hand side is $<1$. Then we can take the last term to the other side and divide by a positive factor, obtaining 
\begin{align}\label{eq:Lipest}
\norm{u-v}_{\gamma,\alpha+\beta}&\leqslant C[\norm{u^{T,\sharp}-v^{T,\sharp}}_{\calC^{(2-\gamma')\alpha+2\beta-1}_p}+\norm{u^{T,\prime}-v^{T,\prime}}_{(\calC^{\alpha+\beta-1}_p)^d}\nonumber\\&\qquad+\norm{f^{\sharp}-g^{\sharp}}_{\mathcal{L}^{\gamma',\alpha+2\beta-1}_T}+\norm{f^{\prime}-g^{\prime}}_{(\mathcal{L}^{\gamma',\alpha+\beta-1}_T)^d}+\norm{\calV-\mathcal{W}}_{\calX^{\beta,\gamma'}}],
\end{align}
where $C=C(T,\norm{\mathcal V},\norm{\mathcal W},\norm{u^{T}},\norm{v^{T}},\norm{f},\norm{g})>0$ is a constant that depends on the norms of the input data.
Thus, the map $(u^{T},f,\mathcal V)\mapsto (u,u^{\sharp})$ is locally Lipschitz continuous, which implies the claim.\\
If $T$ is such that $(T^{\gamma-\gamma'}\vee T^{\gamma'-\gamma''})C$ times the implicit constant is at least $1$, then we want to apply the estimates above on the subintervals $[T-k\overline{T},T-(k-1)\overline{T}]$ of length $\overline{T}$, where $\overline{T}$ is chosen, such that $(\overline{T}^{\gamma-\gamma'}\vee \overline{T}^{\gamma'-\gamma''})C$ times the implicit constant is strictly less than $1$ and where $k=1,\dots,n$ for $n\in\N$ with $T-n\overline{T}\leqslant 0$. To obtain the continuity in $\mathcal{D}^{\gamma}_{T}$, we consider the solutions $u^{[T-k\overline{T},T-(k-1)\overline{T}]},v^{[T-k\overline{T},T-(k-1)\overline{T}]}$ on the subintervals $[T-k\overline{T},T-(k-1)\overline{T}]$ for $k=1,\dots n$, where the terminal condition of the solution $u^{[T-k\overline{T},T-(k-1)\overline{T}]}$ is the initial value of the solution $u^{[T-(k-1)\overline{T},T-(k-2)\overline{T}]}$ (analogously for $v$), such that, patched together, we obtain the solutions $u,v$ on $[0,T]$.\\ Let $\epsilon >0$ to be chosen below.\\ For $k=2,\dots,n$, we have that $u^{[T-k\overline{T},T-(k-1)\overline{T}]},v^{[T-k\overline{T},T-(k-1)\overline{T}]}\in \mathcal{D}_{T-(k-1)\overline{T}}^{0,\alpha+\beta}$ (see the argument in the proof of \cref{thm:singular}), such that we can estimate 
\begin{align}\label{eq:partition1}
\MoveEqLeft
\norm{u-v}_{\mathcal{M}_{T}^{\gamma'}\calC^{\alpha+\beta-\epsilon\alpha}_p}\nonumber\\&\leqslant T^{\gamma'}\norm{u-v}_{\mathcal{M}^{0}_{T-\overline{T}}\calC^{\alpha+\beta-\epsilon\alpha}_p}+\norm{u-v}_{\mathcal{M}_{\overline{T},T}^{\gamma'}\calC^{\alpha+\beta}_p}\nonumber\\&\leqslant T^{\gamma'}\sum_{k=2}^{n}\norm{u-v}_{\mathcal{M}^{0}_{\overline{T},T-(k-1)\overline{T}}\calC^{\alpha+\beta-\epsilon\alpha}_p}+\norm{u-v}_{\mathcal{M}^{\gamma'}_{\overline{T},T}\calC^{\alpha+\beta}_p}.
\end{align} 
Furthermore, we can estimate for $\epsilon\in (0,\gamma']$,
\begin{align}\label{partition2}
\norm{u-v}_{C_{T}^{1-\gamma'}\calC^{\beta}_p}&\leqslant T^{\gamma'-\epsilon}\norm{u-v}_{C_{T-\overline{T}}^{1-\epsilon}\calC^{\beta}_p}+\norm{u-v}_{C_{\overline{T},T}^{1-\gamma'}\calC^{\beta}_p}\nonumber
\\&\leqslant T^{\gamma'-\epsilon}\sum_{k=2}^{n}\norm{u-v}_{\mathcal{L}_{\overline{T},T-(k-1)\overline{T}}^{\epsilon,\beta+\alpha}}+\norm{u-v}_{\mathcal{L}_{\overline{T},T}^{\gamma',\beta+\alpha}}.
\end{align}
Subtracting the terminal condition for each of the terms with $k=2,\dots,n$ and applying the interpolation bound \eqref{eq:i3b} for $\theta=\alpha+\beta$, $\tilde{\theta}=\epsilon\alpha$ yields
for $k=2,\dots,n$,
\begin{align}\label{eq:int-bound}
\MoveEqLeft
\norm{(u-u_{T-(k-1)\overline{T}})-(v-v_{T-(k-1)\overline{T}})}_{\mathcal{M}^{0}_{\overline{T},T-(k-1)\overline{T}}\calC^{\alpha+\beta-\epsilon\alpha}_p}\nonumber\\&\leqslant \norm{(u-u_{T-(k-1)\overline{T}})-(v-v_{T-(k-1)\overline{T}})}_{\mathcal{M}^{\epsilon}_{\overline{T},T-(k-1)\overline{T}}\calC^{\alpha+\beta}_p}.
\end{align}
Together with \eqref{eq:partition1}, \eqref{partition2}  and \eqref{eq:int-bound}, this then yields
\begin{align}\label{eq:partition}
\hspace{0.1em}&\hspace{-0.1em}
\norm{u-v}_{\mathcal{L}_{T}^{\gamma',\alpha+\beta-\epsilon\alpha}}\nonumber
\\&\lesssim \squeeze[1]{T^{\gamma'}\sum_{k=2}^{n}\paren[\big]{\norm{(u-u_{T-(k-1)\overline{T}})-(v-v_{T-(k-1)\overline{T}})}_{\mathcal{L}_{\overline{T},T-(k-1)\overline{T}}^{0,\alpha+\beta-\epsilon\alpha}}+\norm{u_{T-(k-1)\overline{T}}-v_{T-(k-1)\overline{T}}}_{\calC^{\alpha+\beta}_p}}}\nonumber\\&\qquad\quad+\norm{u-v}_{\mathcal{L}_{\overline{T},T}^{\gamma',\alpha+\beta}}\nonumber
\\&\lesssim \squeeze[1]{T^{\gamma'}\sum_{k=2}^{n}\paren[\big]{\norm{(u-u_{T-(k-1)\overline{T}})-(v-v_{T-(k-1)\overline{T}})}_{\mathcal{L}_{\overline{T},T-(k-1)\overline{T}}^{\epsilon,\alpha+\beta}}+\norm{u_{T-(k-1)\overline{T}}-v_{T-(k-1)\overline{T}}}_{\calC^{\alpha+\beta}_p}}}\nonumber\\&\qquad\quad+\norm{u-v}_{\mathcal{L}_{\overline{T},T}^{\gamma',\alpha+\beta}}\nonumber
\\&\lesssim T^{\gamma'}\sum_{k=2}^{n}\norm{u-v}_{\mathcal{L}_{\overline{T},T-(k-1)\overline{T}}^{\epsilon,\alpha+\beta}}+\norm{u-v}_{\mathcal{L}_{\overline{T},T}^{\gamma',\alpha+\beta}},
\end{align}
where in the last estimate, we estimated the norm of the terminal conditions by the norm of the solutions in the previous iteration step.\\
Analogously, we can argue for $u^{\sharp}-v^{\sharp}$, obtaining
\begin{align}\label{eq:sharp-partition}
\MoveEqLeft
\norm{u^{\sharp}-v^{\sharp}}_{\mathcal{L}_{T}^{\gamma,2(\alpha+\beta)-1-\epsilon\alpha}}\nonumber
\\&\lesssim T^{\gamma}\sum_{k=2}^{n}\norm{u^{\sharp}-v^{\sharp}}_{\mathcal{L}_{\overline{T},T-(k-1)\overline{T}}^{\epsilon,2(\alpha+\beta)-1}}+\norm{u^{\sharp}-v^{\sharp}}_{\mathcal{L}_{\overline{T},T}^{\gamma,2(\alpha+\beta)-1}}.
\end{align}
Now, taking $\epsilon:=(\gamma-\gamma')\in (0,\gamma')$, we can apply the above estimate \eqref{eq:Lipest} for each of the terms on the right-hand side of the inequalities \eqref{eq:partition} and \eqref{eq:sharp-partition}. That is, for each of the terms for $k=2,\dots,n$, we obtain
\begin{align*}
\hspace{1.5em}&\hspace{-1.5em}
\norm{u-v}_{\mathcal{L}_{\overline{T},T-(k-1)\overline{T}}^{\epsilon,\alpha+\beta}}+\norm{u^{\sharp}-v^{\sharp}}_{\mathcal{L}_{\overline{T},T-(k-1)\overline{T}}^{\epsilon,2(\alpha+\beta)-1}}\\&\lesssim \frac{1}{\squeeze[1]{1-\overline{T}^{\epsilon}\norm{\calV}_{\calX^{\beta,\gamma'}}(1+\norm{\calV}_{\calX^{\beta,\gamma'}})}}\bigg[\norm{u^{T,\sharp}-v^{T,\sharp}}_{\calC^{(2-\gamma')\alpha+2\beta-1}_p}+\norm{u^{T,\prime}-v^{T,\prime}}_{(\calC^{\alpha+\beta-1}_p)^d}\nonumber\\&\qquad+\norm{f^{\sharp}-g^{\sharp}}_{\mathcal{L}^{\gamma',\alpha+2\beta-1}_T}+\norm{f^{\prime}-g^{\prime}}_{(\mathcal{L}^{\gamma',\alpha+\beta-1}_T)^d}+\norm{\calV-\mathcal{W}}_{\calX^{\beta,\gamma'}}\bigg].
\end{align*}
This uses that by the choice of $\epsilon$, $\overline{T}^{\epsilon}=\overline{T}^{\gamma-\gamma'}\leqslant \overline{T}^{\gamma-\gamma'}\vee \overline{T}^{\gamma'-\gamma''}$ and that $u,v\in \mathcal{D}_{T-(k-1)\overline{T}}^{0,\alpha+\beta}$ for $k=2,\dots,n$.
For $k=1$, we replace $\epsilon$ by $\gamma$, respectively $\gamma'$ for $u^{\sharp}-v^{\sharp}$, and obtain the estimate \eqref{eq:Lipest} on the subinterval $[T-\overline{T},T]$.
Together, this then yields the following estimate on the whole interval $[0,T]$ (with a possibly different constant $C$): 
\begin{align}\label{eq:Lipest-2}
\norm{u-v}_{\gamma,\alpha+\beta-\epsilon\alpha }&\leqslant C[\norm{u^{T,\sharp}-v^{T,\sharp}}_{\calC^{(2-\gamma')\alpha+2\beta-1}_p}+\norm{u^{T,\prime}-v^{T,\prime}}_{(\calC^{\alpha+\beta-1}_p)^d}\nonumber\\&\quad+\norm{f^{\sharp}-g^{\sharp}}_{\mathcal{L}^{\gamma',\alpha+2\beta-1}_{T}}+\norm{f^{\prime}-g^{\prime}}_{(\mathcal{L}^{\gamma',\alpha+\beta-1}_{T})^d}+\norm{\calV-\mathcal{W}}_{\calX^{\beta,\gamma'}}].
\end{align}
Plugging now $u-v$ back in the contraction map on $[0,T]$, we can remove the loss $\epsilon\alpha$ in regularity. That is, we can estimate for $\epsilon$ small enough, 
\begin{align*}
\norm{u-v}_{\gamma,\alpha+\beta}&\lesssim \norm{\calV}_{\calX^{\beta,\gamma'}}(1+\norm{\calV}_{\calX^{\beta,\gamma'}})\norm{u-v}_{\gamma,\alpha+\beta-\alpha\epsilon} \\&\quad + C[\norm{u^{T,\sharp}-v^{T,\sharp}}_{\calC^{(2-\gamma')\alpha+2\beta-1}_p}+\norm{u^{T,\prime}-v^{T,\prime}}_{(\calC^{\alpha+\beta-1}_p)^d}\nonumber\\&\quad+\norm{f^{\sharp}-g^{\sharp}}_{\mathcal{L}^{\gamma',\alpha+2\beta-1}_{T}}+\norm{f^{\prime}-g^{\prime}}_{(\mathcal{L}^{\gamma',\alpha+\beta-1}_{T})^d}+\norm{\calV-\mathcal{W}}_{\calX^{\beta,\gamma'}}].
\end{align*} 
Thus the local Lipschitz continuity \eqref{eq:Lip-est} on $[0,T]$ follows.\\
In the setting of \cref{cor:non-singular}, we obtain from the above, that the Lipschitz estimate \eqref{eq:Lip-est} holds true with $\gamma'=0$ for the norms of $u^{T,\sharp}-v^{T,\sharp},f^{\sharp}-g^{\sharp},f^{\prime}-g^{\prime}$ and $u^{T,\prime}=v^{T,\prime}=0$ on the right-hand side and any small $\gamma>0$ on the left-hand-side of the estimate. Similar as in the proof of \cref{cor:non-singular}, we can use the fixed point property and the estimate \eqref{eq:betterest} for small enough $\gamma>0$, together with the Schauder estimates and the interpolation bound \eqref{eq:i3b}, to obtain that
\begin{align*}
\hspace{0.3em}&\hspace{-0.3em}
\norm{u-v}_{\mathcal{L}_{T}^{0,\alpha+\beta}}+\norm{u^{\sharp}-v^{\sharp}}_{\mathcal{L}_{T}^{0,2(\alpha+\beta)-1}}\\&\lesssim \norm{\calV}_{\calX^{\beta,\gamma'}}(1+\norm{\calV}_{\calX^{\beta,\gamma'}})\paren[\Big]{\norm{u-v}_{\mathcal{L}_{T}^{0,\alpha+\beta-\gamma\alpha}}+\norm{u^{\sharp}-v^{\sharp}}_{\mathcal{L}_{T}^{0,2(\alpha+\beta)-1-\gamma\alpha}}}\\&\quad + \squeeze[1]{C[\norm{u^{T,\sharp}-v^{T,\sharp}}_{\calC^{2\alpha+2\beta-1}_p}+\norm{f^{\sharp}-g^{\sharp}}_{\mathcal{L}^{0,\alpha+2\beta-1}_{T}}+\norm{f^{\prime}-g^{\prime}}_{(\mathcal{L}^{0,\alpha+\beta-1}_{T})^d}+\norm{\calV-\mathcal{W}}_{\calX^{\beta,\gamma'}}]}
\\&\lesssim \norm{\calV}_{\calX^{\beta,\gamma'}}(1+\norm{\calV}_{\calX^{\beta,\gamma'}})\times\\&\qquad\qquad\quad\squeeze[1]{\paren[\Big]{\norm{(u-v)-(u^{T,\sharp}-v^{T,\sharp})}_{\mathcal{L}_{T}^{0,\alpha+\beta-\gamma\alpha}}+\norm{(u^{\sharp}-v^{\sharp})-(u^{T,\sharp}-v^{T,\sharp})}_{\mathcal{L}_{T}^{0,2(\alpha+\beta)-1-\gamma\alpha}}}}\\&\quad + \squeeze[1]{C[\norm{u^{T,\sharp}-v^{T,\sharp}}_{\calC^{2\alpha+2\beta-1}_p}+\norm{f^{\sharp}-g^{\sharp}}_{\mathcal{L}^{0,\alpha+2\beta-1}_{T}}+\norm{f^{\prime}-g^{\prime}}_{(\mathcal{L}^{0,\alpha+\beta-1}_{T})^d}+\norm{\calV-\mathcal{W}}_{\calX^{\beta,\gamma'}}]}
\\&\lesssim \norm{\calV}_{\calX^{\beta,\gamma'}}(1+\norm{\calV}_{\calX^{\beta,\gamma'}})\paren[\Big]{\norm{u-v}_{\mathcal{L}_{T}^{\gamma,\alpha+\beta}}+\norm{u^{\sharp}-v^{\sharp}}_{\mathcal{L}_{T}^{\gamma,2(\alpha+\beta)-1}}}\\&\quad +\squeeze[1]{ C[\norm{u^{T,\sharp}-v^{T,\sharp}}_{\calC^{2\alpha+2\beta-1}_p}+\norm{f^{\sharp}-g^{\sharp}}_{\mathcal{L}^{0,\alpha+2\beta-1}_{T}}+\norm{f^{\prime}-g^{\prime}}_{(\mathcal{L}^{0,\alpha+\beta-1}_{T})^d}+\norm{\calV-\mathcal{W}}_{\calX^{\beta,\gamma'}}].}
\end{align*}
Notice that, to apply the interpolation bound \eqref{eq:i3b} in the last estimate above, we subtracted the terminal condition $u^{T}-v^{T}=u^{T,\sharp}-v^{T,\sharp}$, so that $(u-v)_{T}-(u^{T}-v^{T})=0$. The constant $C$ above changes in each line.
Thus together the Lipschitz continuity of the solution map with values in $\mathcal{L}_{T}^{0,\alpha+\beta}\times \mathcal{L}_{T}^{0,2(\alpha+\beta)-1}$ follows. 
\end{proof}

\begin{remark}[Super-exponential dependency of the Lipschitz constant on $\mathcal{V},\mathcal{W}$]
The Lipschitz constant of the solution map on $[0,T]$ depends super-exponentially on the norms $\norm{\mathcal{V}}_{\mathcal{X}^{\beta,\gamma'}}$, $\norm{\mathcal{W}}_{\mathcal{X}^{\beta,\gamma'}}$. Indeed, to obtain the Lipschitz estimate of the solution map on $[0,T]$, we have to apply the estimate in \eqref{eq:rgeneq-pr1} on every subinterval $[T-(k+1)\overline{T},T-k\overline{T}]$, where we have to choose $\overline{T}$ small enough so that $\overline{T}^\kappa < C^{-1}$ for $\kappa = \gamma-\gamma' \wedge \gamma'-\gamma''$ and for the constant $C=C(\norm{\mathcal V},\norm{\mathcal W},\norm{u^{T}},\norm{v^{T}},\norm{f},\norm{g})$. This means that in \eqref{eq:Lipest} we have to iterate the estimate at least $T/C^{-\kappa} = T C^\kappa$ times, and each time we multiply with the constant $\tilde C$ in~\eqref{eq:Lipest}, leading roughly speaking to a factor $\tilde C^{T C^\kappa}$. By doing the analysis more carefully we can show that there is (super-)exponential dependence only on $\norm{\mathcal V},\norm{\mathcal W}$ and that the Lipschitz constant actually depends linearly on $\norm{u^{T}},\norm{v^{T}},\norm{f},\norm{g}$. But the super-exponential dependence on $\norm{\mathcal V},\norm{\mathcal W}$ is inherent to the problem and we expect that it cannot be significantly improved. By similar arguments, we also see that the norm of the solution $u$ to the Kolmogorov backward equation in \cref{thm:singular} depends super-exponentially on $\|\mathcal V\|$.

\noindent This will be relevant when we take $\mathcal{V}$ random, cf.~the Brox diffusion with Lévy noise in \cite{kp}. If we do not have super-exponential moments for $\norm{\mathcal{V}}_{\mathcal{X}^{\beta,\gamma'}}$, then we do not know if $u$ has finite moments. And if $V$ is Gaussian, then the second component of the lift $\mathcal V$ is a second order polynomial of a Gaussian and therefore it does not have super-exponential moments.
	
\noindent But note that this only concerns H\"older norms of the Kolmogorov backward equation. If we are only interested in the $L^{p}$ norm, $p\in[1,\infty]$, we can always use the trivial bound $\|u\|_{L^{p}} \le \|u^T\|_{L^{p}} + T \|f\|_{L^{p}}$, provided that the right-hand side is finite, which holds for smooth $V,f,u^T$ by the stochastic representation (Feynman-Kac) of the Kolmogorov backward equation, and which extends by approximation to the general setting.
\end{remark}

\noindent Next we consider solutions of the Kolmorov PDE for $\mathcal{G}^{\calV}$ (for fixed $\calV$) on subintervals $[0,r]$ of $[0,T]$ for bounded sets of terminal conditions $(y^{r})_{r\in [0,T]}$ and right-hand-sides $(f^{r})_{r\in[0,T]}$. In certain situations one is interested in a uniform bound on the norms of the solutions $(u^{r})$ on $[0,r]$. We prove the latter in the following corollary.\\
The solution $u^{r}$ on $[0,r]$ has the following paracontrolled structure
\begin{align}\label{eq:para-r}
u^{r}=u^{r,\sharp}+(\nabla u^{r}-f^{r,\prime})\para J^{r}(V)+y^{r,\prime}\para P_{r-\cdot}V_{r}
\end{align}
with
\begin{align*}
u^{r,\sharp}&=P_{r-\cdot}y^{r,\sharp}+J^{r}(-f^{\sharp})+J^{r}(\nabla u^{r}\reso V)+J^{r}(V\para \nabla u^{r})\\&\quad+C_{1}(y^{r,\prime},V_{r})+C_{2}(-f^{\prime},V)+C_{2}(\nabla u^{r},V),
\end{align*}
for the commutators from the proof of \cref{thm:singular}.
\begin{corollary}\label{thm:uniform-singular}
Let $T>0$ and $\calV\in\calX^{\beta,\gamma'}$ for $\beta,\gamma'$ as in \cref{thm:singular}. Let $\gamma\in (\gamma',1)$ and $\gamma''\in (0,\gamma')$ be as in the proof of \cref{thm:singular}. 
Let $(y^{r}=y^{r,\sharp}+y^{r,\prime}\para V_{r})_{r\in [0,T]}$ be a bounded sequence of singular paracontrolled terminal conditions, that is,
\begin{align*}
C_{y}:=\sup_{r\in[0,T]}[\norm{y^{r,\sharp}}_{\calC^{(2-\gamma')\alpha+2\beta-1}_p}+ \norm{y^{r,\prime}}_{\calC^{\alpha+\beta-1}_p}]<\infty.
\end{align*} 
Let $(f^{r}=f^{r,\sharp}+f^{r,\prime}\para V)_{r\in[0,T]}$ be a sequence of right-hand-sides with 
\begin{align*}
C_{f}:=\sup_{r\in[0,T]}[\norm{f^{r,\sharp}}_{\mathcal{L}_{r}^{\gamma',\alpha+2\beta-1}}+\norm{f^{r,\prime}}_{\mathcal{L}_{r}^{\gamma',\alpha+\beta-1}}]<\infty.
\end{align*} 
Let for $r\in [0,T]$, $(u^{r}_{t})_{t\in [0,r]}$ be the solution of the backward Kolmogorov PDE for $\mathcal{G}^{\calV}$ with terminal condition $u^{r}_{r}=y^{r}$ and right-hand side $f^{r}$.\\ 
Then, the following uniform bound for the solutions $(u^{r})$ holds true
\begin{align}\label{eq:uniform-singular}
\MoveEqLeft
\sup_{r\in[0,T]}[\norm{u^{r,\sharp}}_{\mathcal{L}_{r}^{\gamma,2(\alpha+\beta)-1}}+\norm{u^{r}}_{\mathcal{L}_{r}^{\gamma',\alpha+\beta}}]\nonumber\\&\lesssim_{T} 
\lambda_{\overline{T},\calV}^{-1}\paren[\bigg]{\sup_{r\in[0,T]}[\norm{y^{r,\sharp}}_{\calC^{(2-\gamma')\alpha+2\beta-1}_p}+\norm{f^{r,\sharp}}_{\mathcal{L}_{r}^{\gamma',\alpha+2\beta-1}}] \nonumber\\&\qquad\qquad+\norm{\calV}_{\calX^{\beta,\gamma'}}\sup_{r\in[0,T]}[\norm{y^{r,\prime}}_{\calC^{\alpha+\beta-1}_p}+\norm{f^{r,\prime}}_{\mathcal{L}_{r}^{\gamma',\alpha+\beta-1}}]},
\end{align} 
where $\lambda_{\overline{T},\calV}:=1-(\overline{T}^{\gamma-\gamma'}\vee\overline{T}^{ \gamma'-\gamma''})\norm{\calV}_{\calX^{\beta,\gamma'}}(1+\norm{\calV}_{\calX^{\beta,\gamma'}})>0$.\\
In particular, replacing $y^{r}$ by $y^{r}_{1}-y^{r}_{2}$ and $f^{r}$ by $f^{r}_{1}-f^{r}_{2}$ with analogue bounds, a uniform Lipschitz bound for the solutions $u^{r}_{1}-u^{r}_{2}$ follows.\\
In the setting of \cref{cor:non-singular}, the bound \eqref{eq:uniform-singular} holds true with $\gamma=\gamma'=0$, under the assumption, that $C_{f}+C_{y}<\infty$ for $\gamma'=0$.
\end{corollary}
\begin{remark}\label{rem:uniform-y}
In setting of \cref{thm:ygeneq} for $\beta$ in the Young regime and considering bounded sets of terminal conditions $(y^{r})_{r}\subset \calC^{(1-\gamma)\alpha+\beta}_p$ and right-hand-sides $f^{r}\subset\mathcal{L}_{r}^{\gamma,\beta}$ for $\gamma\in [0,1)$, an analogue uniform Lipschitz bound for the solutions $(u^{r})$ on $[0,r]$ holds true. The proof is similar except much easier.
\end{remark}
\begin{proof}
The proof follows from \cref{thm:cont} replacing $T$ by $r$ and considering paracontrolled solutions on $[0,r]$ in the sense of \eqref{eq:para-r}. Then, by \eqref{eq:product-bound} and \eqref{eq:rgeneq-pr1} from the proof of \cref{thm:cont} for $\calV=\mathcal{W}$ and splitting the interval $[0,r]$ in subintervals of length $\overline{T}$, we obtain for every $r\leqslant T$,
\begin{align*}
\MoveEqLeft
\norm{u^{r,\sharp}}_{\mathcal{L}_{r}^{\gamma,2(\alpha+\beta)-1}}+\norm{u^{r}}_{\mathcal{L}_{r}^{\gamma',\alpha+\beta}}\\&\lesssim_{T} C(r)\lambda_{\overline{T},\calV}^{-1}\paren[\bigg]{\sup_{r\in[0,T]}[\norm{y^{r,\sharp}}_{\calC^{(2-\gamma')\alpha+2\beta-1}_p}+\norm{f^{r,\sharp}}_{\mathcal{L}_{r}^{\gamma',\alpha+2\beta-1}}] \nonumber\\&\qquad\qquad\qquad+\norm{\calV}_{\calX^{\beta,\gamma'}}\sup_{r\in[0,T]}[\norm{y^{r,\prime}}_{\calC^{\alpha+\beta-1}_p}+\norm{f^{r,\prime}}_{\mathcal{L}_{r}^{\gamma',\alpha+\beta-1}}]}.
\end{align*}
The dependence of the constant $C(r)$ on $r\leqslant T$ is as follows: $C(r)\lesssim \frac{r}{\overline{T}}\leqslant\frac{T}{\overline{T}}$. Notice that the choice of $\overline{T}$ only depends on $\norm{\calV}$, which is fixed here. Thus we obtain \eqref{eq:uniform-singular}. As the solution $u^{r}$ depends linearily on the terminal condition $y^{r}$ and the right-hand side $f^{r}$, the uniform Lipschitz bound follows.
\end{proof}
\begin{remark}\label{rem:lin-in-y-f}
Let $(V^{m})$ be such that $\calV^{m}:=(V^{m},(\sum_{i}P(\partial_{i}V^{m,j})\reso V^{i})_{j})\stackrel{m \to\infty}{\to} \calV $ in $\calX^{\beta,\gamma'}$. Let $(f^{r})$, $(y^{r})$ be as in the corollary. Moreover, let  $(y^{r,m})$ with $y^{r,m}=y^{r,\sharp,m}+y^{r,\prime}\para V_{r}^{m}$ be such that $\sup_{r\in [0,T]}\norm{y^{r,\sharp,m}-y^{r,\sharp}}_{\calC_{p}^{(2-\gamma')\alpha+2\beta-1}}\to 0$ for $m\to\infty$. Analogously, let $(f^{r,m})$ with $f^{r,m}=f^{r,\sharp,m}+f^{r,\prime}\para V^{m}$ and convergence of $(f^{r,\sharp,m})_m$. Let $u^{r}$ and $u^{r,m}$ be the solutions for $\mathcal{G}^{\calV}$ with right-hand side $f^{r}$ and terminal conditions $y^{r}$ and $y^{r,m}$, respectively.
Then the proof of the corollary furthermore shows that
\begin{align*}
\MoveEqLeft
\sup_{r\in[0,T]}[\norm{u^{r,\sharp}-u^{r,\sharp,m}}_{\mathcal{L}_{r}^{\gamma,2(\alpha+\beta)-1}}+\norm{u^{r}-u^{r,m}}_{\mathcal{L}_{r}^{\gamma',\alpha+\beta}}]\nonumber\\&\lesssim_{T} 
\lambda_{\overline{T},\calV}^{-1}\paren[\big]{\sup_{r\in[0,T]}[\norm{y^{r,\sharp}-y^{r,\sharp,m}}_{\calC^{(2-\gamma')\alpha+2\beta-1}_p}+\norm{f^{r,\sharp}-f^{r,\sharp,m}}_{\mathcal{L}_{r}^{\gamma',\alpha+2\beta-1}}] \nonumber\\&\qquad\qquad+\norm{\calV-\calV^{m}}_{\calX^{\beta,\gamma'}}\sup_{r\in[0,T]}[\norm{y^{r,\prime}}_{\calC^{\alpha+\beta-1}_p}+\norm{f^{r,\prime}}_{\mathcal{L}_{r}^{\gamma',\alpha+\beta-1}}]}\\&\to 0,
\end{align*} for $m\to\infty$.
\end{remark}

\end{section}

\appendix
\begin{section}{Appendix}\label{Appendix A}
\begin{proof}[Proof of \cref{lem:La-comm}]
The proof of the lemma uses ideas from the proof of \cite[Lemma 5.3.20]{doktorp}. Let $\psi=p_{0}\in C_{c}^{\infty}$ and let $j\geqslant 0$ (for $j=-1$, $\Delta_{j}(f\para g)=0$, so there is nothing to estimate).  
Then we estimate (notation: $S_{j-1}u:=\sum_{l\leqslant j-1}\Delta_{l}u$)
\begin{align*}
\MoveEqLeft
\norm{\Delta_{j}[(-\La)(f\para g)-f\para (-\La)g]}_{L^{p}}\\&=\paren[\bigg]{\int_{\R^{d}}\abs[\bigg]{\int_{\R^{d}}\F^{-1}(-\psi_{\nu}^{\alpha}p_{j})(x-y)(S_{j-1}f(y)-S_{j-1}f(x))\Delta_{j}g(y)dy}^{p}dx}^{1/p}
\\&\lesssim \sum_{\abs{\eta}=1}\norm{ [z\mapsto z^{\eta}\F^{-1}(-\psi_{\nu}^{\alpha}p_{j})(z)]\ast \partial^{\eta}(S_{j-1}f) }_{L^{p}}\norm{\Delta_{j}g}_{L^{\infty}}
\\&\lesssim\sum_{\abs{\eta}=1}\norm{z\mapsto z^{\eta}\F^{-1}(-\psi^{\alpha}_{\nu}p_{j})(z)}_{L^{1}}\norm{\partial^{\eta}S_{j-1}f}_{L^{p}}\norm{\Delta_{j}g}_{L^{\infty}}
\end{align*} for a multi-index $\eta$ and using that $S_{j-1}f(x)-S_{j-1}f(y)=\int_{0}^{1}DS_{j-1}f(\lambda x+(1-\lambda)y))(x-y)d\lambda$ with $\lambda x+(1-\lambda)y = (1+\lambda)x-\lambda y - (x-y)$ (and substituting $y\to x-y$, $x\to (1+\lambda)x-\lambda y$) and Young's inequality for the last estimate. We have that, as $\sigma<1$, 
\begin{align*}
\norm{\partial^{\eta}S_{j-1}f}_{L^{p}}\norm{\Delta_{j}g}_{L^{\infty}}\lesssim 2^{-j(\sigma-1+\varsigma)}\norm{\partial^{\eta}f}_{\calC^{\sigma-1}_{p}}\norm{g}_{\calC^{\varsigma}}\lesssim2^{-j(\sigma-1+\varsigma)}\norm{f}_{\calC^{\sigma}_{p}}\norm{g}_{\calC^{\varsigma}}.
\end{align*} 
Moreover, we obtain
\begin{align*}
\norm{z\mapsto z^{\eta}\F^{-1}(-\psi^{\alpha}_{\nu}p_{j})(z)}_{L^{1}}&=2^{j\alpha}\norm{z\mapsto z^{\eta}\F^{-1}(\psi^{\alpha}_{\nu}(2^{-j}\cdot)p_{0}(2^{-j}\cdot))(z)}_{L^{1}}\\&=2^{j\alpha}2^{-j}\norm{\F^{-1}(\partial^{\eta}[\psi^{\alpha}_{\nu}p_{0}](2^{-j}\cdot))}_{L^{1}}\\&\lesssim 2^{j(\alpha-1)}
\end{align*} using that 
\begin{align*}
\norm{\F^{-1}(\partial^{\eta}[\psi^{\alpha}_{\nu}p_{0}](2^{-j}\cdot))}_{ L^{1}}=\norm{2^{jd}\F^{-1}(\partial^{\eta}[\psi^{\alpha}_{\nu}p_{0}])(2^{-j}\cdot)}_{ L^{1}}=\norm{\F^{-1}(\partial^{\eta}[\psi^{\alpha}_{\nu}p_{0}])}_{ L^{1}}<\infty.
\end{align*} 
Together we have 
\begin{align*}
\norm{\Delta_{j}[(-\La)(f\para g)-f\para (-\La)g]}_{L^{p}}&\lesssim 2^{-j(\sigma+\varsigma-\alpha)}\norm{f}_{\calC^{\sigma}_{p}}\norm{g}_{\calC^{\varsigma}},
\end{align*} which yields the claim.
\end{proof}

\begin{proof}[Proof of \cref{schaudercom}]
For $\vartheta\in [-1,\infty)$, the claim follows from \cite[Lemma 5.3.20 and Lemma 5.5.7]{doktorp}, applied to $\varphi(z)=\exp(-\psi^{\alpha}_{\nu}(z))$. Here, \cite[Lemma 5.3.20]{doktorp} can be generalized, with the notation from that lemma, to $u\in\calC^{\alpha}_{p}$ for $p\in[1,\infty]$ arguing analoguously as in the proof of \cref{lem:La-comm}.\\
It remains to prove the commutator for $\vartheta\in[-\alpha,-1)$. For that we note that 
\begin{align*}
P_{t}( u\para v)- u\para P_{t}(v)&=(P_{t}-\operatorname{Id})(u\para v) -u \para (P_{t}-\operatorname{Id})v\\&=\int_{0}^{t}[(-\La) P_{r}(u\para v) -u \para (-\La) P_{r}v]dr.
\end{align*}
For the operator $(-\La)P_{r}$ we have by \cref{lem:a1} (whose claim follows from \eqref{eq:scom} for $\vartheta\geqslant 0$ and \cref{lem:La-comm}), that for $\theta\geqslant 0$ (uniformly in $r\in[0,t]$)
\begin{align*}
\norm{(-\La)P_{r}(u\para v)-u\para (-\La)P_{r}v}_{\calC^{\sigma+\varsigma-\alpha+\theta}_p}\lesssim r^{-\theta/\alpha}\norm{u}_{\calC^{\sigma}_p}\norm{v}_{\calC^{\varsigma}}
\end{align*} holds true and thus we obtain (taking $\theta=\vartheta+\alpha\geqslant 0$)
\begin{align*}
\MoveEqLeft
\norm{P_{t}( u\para v)- u\para P_{t}(v)}_{\calC^{\varsigma+\sigma+\vartheta}_p}\\&\leqslant\int_{0}^{t}\norm{(-\La) P_{r}(u\para v) -u \para (-\La) P_{r}v}_{\calC^{(\varsigma+\sigma-\alpha)+(\vartheta+\alpha)}_p}dr\\&\lesssim  \norm{u}_{\calC^{\sigma}_p}\norm{v}_{\calC^\varsigma}\int_{0}^{t}r^{-(\vartheta+\alpha)/\alpha}dr\lesssim t^{-\vartheta/\alpha}\norm{u}_{\calC^{\sigma}_p}\norm{v}_{\calC^\varsigma},
\end{align*} where the last two estimates are valid for $\vartheta\in [-\alpha,0)$.
\end{proof}

\begin{proof}[Proof of \cref{lem:a1}]
We have that 
\begin{align*}
(-\La)P_{t}(u\para v)-u\para (-\La)P_{t}v &= (-\La)\paren[\big]{P_{t}(u\para v)-u\para P_{t}v}\\&\quad+(-\La)(u\para P_{t}v)-u\para (-\La)P_{t}v.
\end{align*} 
The first summand, we estimate by the commutator for $(P_{t})$ from \cref{schaudercom}, 
and continuity of the operator $(-\La)$ from \cref{prop:contfl}, which gives
\begin{align*}
\norm{(-\La)\paren[\big]{P_{t}(u\para v)-u\para P_{t}v}}_{\calC^{\sigma+\varsigma+\theta-\alpha}_{p}}&\lesssim\norm{P_{t}(u\para v)-u\para P_{t}v}_{\calC^{\sigma+\varsigma+\theta}_{p}}\\&\lesssim t^{-\vartheta/\alpha}\norm{u}_{\calC^{\sigma}_{p}}\norm{v}_{\calC^{\varsigma}}.
\end{align*}
The second summand follows from the commutator for $(-\La)$. If $\alpha=2$, then the estimate is immediate due to Leibnitz rule, $\sigma<1$ and Schauder estimates for $P_{t}$ as $\theta\geqslant 0$. If $\alpha\in (1,2)$, then we apply \cref{lem:La-comm} with $f=u$ and $g=P_{t}v$ 
and use the Schauder estimates with $\theta\geqslant 0$, \cref{schauder}, to obtain
\begin{align*}
\norm{(-\La)(u\para P_{t}v)-u\para (-\La)P_{t}v}_{\calC^{\sigma+\varsigma-\alpha+\theta}_{p}}\lesssim \norm{u}_{\calC^{\sigma}_{p}}\norm{P_{t}v}_{\calC^{\varsigma+\theta}}\lesssim t^{-\theta/\alpha}\norm{u}_{\calC^{\sigma}_{p}}\norm{v}_{\calC^{\varsigma}}.
\end{align*} 
Altogether, we obtain the desired bound. 
\end{proof}

\begin{proof}[Proof of \cref{lem:maxreg}]
The proof of the lemma uses the ideas from the proof of \cite[Lemma A.9]{Gubinelli2015Paracontrolled}.
Let $\delta\in (0,\frac{T-t}{2})$ to be chosen later. Then we have that for $j\geqslant -1$ 
\begin{align*}
\Delta_{j}\int_{t}^{T}f_{t,r}dr=\int_{t}^{T}\Delta_{j}f_{t,r}dr=\int_{t+\delta}^{T}\Delta_{j}f_{t,r}dr+\int_{t}^{t+\delta}\Delta_{j}f_{t,r}dr.
\end{align*}
The first summand we estimate as follows, using Minkowski's inequality,
\begin{align*}
\norm[\bigg]{\int_{t+\delta}^{T}\Delta_{j}f_{t,r}dr}_{L^{p}}&\leqslant\int_{t+\delta}^{T}\norm{\Delta_{j}f_{t,r}}_{L^{p}}dr\\&\leqslant C 2^{-j(\sigma+\varsigma+\epsilon\varsigma)}\int_{t+\delta}^{T} (T-r)^{-\gamma}(r-t)^{-(1+\epsilon)}dr\\& = C 2^{-j(\sigma+\varsigma+\epsilon\varsigma)} (T-t)^{-\gamma-\epsilon}\int_{\delta/(T-t)}^{1}(1-r)^{-\gamma}r^{-(1+\epsilon)}dr\\&\leqslant [2 \max(\epsilon^{-1},(1-\gamma)^{-1}) C] 2^{-j(\sigma+\varsigma+\epsilon\varsigma)} (T-t)^{-\gamma-\epsilon}(\delta/(T-t))^{-\epsilon}\\& = [2 \max(\epsilon^{-1},(1-\gamma)^{-1})C] 2^{-j(\sigma+\varsigma)} (T-t)^{-\gamma}(2^{j\varsigma}\delta)^{-\epsilon},
\end{align*} where we used that for $\sigma\in(0,\frac{1}{2})$, as $\epsilon>0$ and $\gamma<1$,
\begin{align*}
\int_{\sigma}^{1}(1-r)^{-\gamma}r^{-(1+\epsilon)}dr&= \int_{\sigma}^{1/2}(1-r)^{-\gamma}r^{-(1+\epsilon)}dr+\int_{1/2}^{1}(1-r)^{-\gamma}r^{-(1+\epsilon)}dr\\&\leqslant [(\frac{1}{2})^{-\gamma}\epsilon^{-1}+(\frac{1}{2})^{-\gamma}(1-\gamma)^{-1}]\sigma^{-\epsilon}\leqslant 2 \max(\epsilon^{-1},(1-\gamma)^{-1})\sigma^{-\epsilon}.
\end{align*} 
For the second summand, we have
\begin{align*}
\norm[\bigg]{\int_{t}^{t+\delta}\Delta_{j}f_{t,r}dr}_{L^{p}}&\leqslant C 2^{-j\sigma}\int_{t}^{t+\delta}(T-r)^{-\gamma}dr\\&= \frac{C}{1-\gamma} 2^{-j\sigma}[(T-t)^{1-\gamma}-(T-t-\delta)^{1-\gamma}]\\&= \frac{C}{1-\gamma} 2^{-j\sigma}(T-t)^{-\gamma}[(T-t)-(T-t-\delta)\paren[\big]{\frac{T-t}{T-t-\delta}}^{\gamma}]\\&\leqslant \frac{C}{1-\gamma}2^{-j\sigma}(T-t)^{-\gamma}\delta.
\end{align*}
The goal is to estimate $\sup_{j\geqslant -1} 2^{j(\sigma+\varsigma)}\norm[\big]{\Delta_{j}\int_{t}^{T}f_{t,r}dr}_{L^{p}}$. For that purpose, we use for $j$ such that $2^{-j\varsigma}\leqslant \frac{T-t}{2}$ the above estimates for $\delta=2^{-j\varsigma}$. If $j$ is such that $2^{-j\varsigma}>\frac{T-t}{2}$, then we trivially estimate
\begin{align*}
\norm[\bigg]{\Delta_{j}\int_{t}^{T}f_{t,r}dr}_{L^{p}}\leqslant C 2^{-j\sigma}\int_{t}^{T}(T-r)^{-\gamma}dr&= \frac{C}{1-\gamma} 2^{-j\sigma}(T-t)^{1-\gamma}\\&\leqslant \frac{C}{1-\gamma} 2^{-j(\sigma+\varsigma)}(T-t)^{-\gamma}
\end{align*} using $\gamma<1$. Together we thus obtain uniformly in $t\in[0,T]$
\begin{align*}
\sup_{j\geqslant -1} 2^{j(\sigma+\varsigma)}\norm[\bigg]{\Delta_{j}\int_{t}^{T}f_{t,r}dr}_{L^{p}}\leqslant [2 \max(\epsilon^{-1},(1-\gamma)^{-1}) C] (T-t)^{-\gamma},
\end{align*} which yields the claim.
\end{proof}
\end{section}

\section*{Acknowledgements}

H.K.~is supported by the Austrian Science Fund (FWF) Stand-Alone programme P 34992. Part of the work was done when H.K. was employed at Freie Universität Berlin and funded by the DFG under Germany's Excellence Strategy - The Berlin Mathematics Research Center MATH+ (EXC-2046/1, project ID: 390685689). N.P.~ gratefully acknowledges financial support by the DFG via Research Unit FOR2402.

\end{document}